\RequirePackage{etex}
\RequirePackage{easybmat}
\documentclass[]{article}
\usepackage{minitoc}
\usepackage[titletoc]{appendix}
\usepackage[utf8]{inputenc}
\usepackage[T1]{fontenc}
\usepackage{graphicx}
\usepackage{amsmath}
\usepackage{listings}
\usepackage{moreverb}
\usepackage{eurosym}
\usepackage{fancyvrb}
\usepackage{afterpage}
\usepackage[final]{pdfpages}
\usepackage[nottoc]{tocbibind}
\usepackage{listings}
\usepackage{amsthm}
\usepackage{amsfonts}
\usepackage{geometry}
\usepackage[]{algorithm2e}
\usepackage{hyperref}
\usepackage{fixltx2e}
\usepackage{amssymb}
\usepackage{tikz}
\usepackage{tkz-graph}
\usetikzlibrary{calc}
\usepackage{float}
\usepackage{color}
\usepackage{lipsum}

\newcommand*{\mathcolor}{}
\def\mathcolor#1#{\mathcoloraux{#1}}

\newtheorem{theorem}{Theorem}[section]
\newtheorem{lemma}[theorem]{Lemma}
\newtheorem{proposition}[theorem]{Proposition}
\newtheorem{cor}[theorem]{Corollary}
\newtheorem{definition}[theorem]{Definition}
\newtheorem{example}[theorem]{Example}

\newtheorem{remark}[theorem]{Remark}

\newtheorem{convention}[theorem]{Convention}

\newcommand\encircle[1]{%
  \tikz[baseline=(X.base)] 
    \node (X) [draw, shape=circle, inner sep=0] {\strut #1};}

\makeatletter
\def\v@rt#1#2{\m@th\ooalign{$\hfil#1|\hfil$\crcr$#1#2$}}
\def\captr{\mathrel{\mathpalette\v@rt\cap}}
\makeatother

\makeatletter

\newcommand{\Rmnum}[1]{\expandafter\@slowromancap\romannumeral #1@}
\makeatother

\begin{document}

\title{Interval Garside Structures for the Complex Braid Groups $B(e,e,n)$}
\author{Georges Neaime\\
Laboratoire de Mathématiques Nicolas Oresme\\
Universit\'e de Caen Normandie\\
\texttt{georges.neaime@unicaen.fr}
}

\date{September 25, 2018}

\maketitle

\begin{abstract}

We define geodesic normal forms for the general series of complex reflection groups $G(e,e,n)$. This requires the elaboration of a combinatorial technique in order to explicitly determine minimal word representatives of the elements of $G(e,e,n)$ over the generating set of the presentation of Corran-Picantin. Using these geodesic normal forms, we construct intervals in $G(e,e,n)$ that are lattices. This gives rise to interval Garside groups. We determine which of these groups are isomorphic to the complex braid group $B(e,e,n)$ and get a complete classification. For the other Garside groups that appear in our construction, we provide some of their properties and compute their second integral homology groups in order to understand these new structures.

\end{abstract}

\begin{small}
\tableofcontents 
\end{small}

\newpage

\section{Introduction}\label{SectionIntroduction}

A complex reflection is a linear transformation of finite order, which fixes a hyperplane pointwise. Let $W$ be a finite subgroup of $GL_n(\mathbb{C})$ with $n \geq 1$ and $\mathcal{R}$ be the set of complex reflections of $W$. We say that $W$ is a complex reflection group if $W$ is generated by $\mathcal{R}$. It is well known that every complex reflection group is a direct product of irreducible ones. These irreducible complex reflection groups have been classified by Shephard and Todd \cite{ShephardTodd} in 1954. The classification consists of the following cases:
\begin{itemize}
\item The infinite series $G(de,e,n)$ depending on three positive integer parameters,
\item $34$ exceptional groups.
\end{itemize}

As we are interested in the groups of the infinite series, we provide the definition of the group $G(de,e,n)$. For the definition of the $34$ exceptional groups, see \cite{ShephardTodd}.
\begin{definition}\label{DefinitionG(de,e,n)}
The group $G(de,e,n)$ is defined as the group of $n \times n$ monomial matrices (each row and column has a unique nonzero entry), where
\begin{itemize}
\item all nonzero entries are $de$-th roots of unity and
\item the product of all the nonzero entries is a $d$-th root of unity.
\end{itemize} 
\end{definition}

Brou\'e, Malle and Rouquier \cite{BMR} managed to associate a complex braid group $B$ to each complex reflection group $W$. The definition of the complex braid group related to $W$ is as follows. Let $\mathcal{A} := \{Ker(s-1)$ s.t. $s \in \mathcal{R} \}$ be the hyperplane arrangement and $X := \mathbb{C}^n\setminus\bigcup\mathcal{A}$ be the hyperplane complement. The complex reflection group $W$ acts naturally on $X$. Let $X/W$ be its space of orbits. 
\begin{definition}
The complex braid group associated with $W$ is defined as the fundamental group:  $$B:=\pi_1(X/W).$$
\end{definition}

If $W$ is equal to $G(de,e,n)$, then we denote by $B(de,e,n)$ the associated complex braid group. According to the results in \cite{BMR}, one can readily check that the complex braid groups $B(de,e,n)$ are all isomorphic to $B(2e,e,n)$ for $d > 1$. Hence the complex braid groups associated with the $3$-parameter series $G(de,e,n)$ arise from the two $2$-parameter series $G(e,e,n)$ and $G(2e,e,n)$. This paper concerns the complex braid groups $B(e,e,n)$ and their associated complex reflection groups $G(e,e,n)$.\\

Let $W$ be a real reflection group, meaning that $W$ is a subgroup of $GL_n(\mathbb{R})$. By \cite{BrieskornArtinTitsFundamentalGroup} and \cite{Bourbaki}, we recover in the previous definitions the notion of Coxeter and Artin-Tits groups that we recall now.

\begin{definition}\label{DefinitionCoxeterGroups}

Assume that $W$ is a group and $S$ be a subset of $W$. For $s$ and $t$ in $S$, let $m_{st}$ be the order of $st$ if this order is finite, and be $\infty$ otherwise. We say that $(W,S)$ is a Coxeter system, and that $W$ is a Coxeter group, if $W$ admits the presentation with generating set $S$ and relations: 
\begin{itemize}
\item quadratic relations: $s^2=1$ for all $s \in S$ and
\item braid relations: $\underset{m_{st}}{\underbrace{sts\cdots}}=\underset{m_{st}}{\underbrace{tst\cdots}}$  for $s,t \in S$, $s \neq t$ and $m_{st} \neq \infty$.
\end{itemize}
\end{definition}

We define the Artin-Tits group $B(W)$ associated with a Coxeter system $(W,S)$ as follows.

\begin{definition}\label{DefinitionArtinTits}
The Artin-Tits group $B(W)$ associated with a Coxeter system $(W,S)$ is defined by a presentation with generating set $\widetilde{S}$ in bijection with the generating set $S$ of the Coxeter group and the relations are only the braid relations: $\underset{m_{st}}{\underbrace{\tilde{s}\tilde{t}\tilde{s}\cdots}}=\underset{m_{st}}{\underbrace{\tilde{t}\tilde{s}\tilde{t}\cdots}}$ for $\tilde{s}, \tilde{t} \in \widetilde{S}$ and $\tilde{s} \neq \tilde{t}$, where $m_{st} \in \mathbb{Z}_{\geq 2}$ is the order of $st$ in $W$.
\end{definition}

Consider $W = S_n$, the symmetric group with $n \geq 2$. It is a Coxeter group with set of generators $S$ consisting of simple transpositions in $S_n$. The Artin-Tits group associated with the Coxeter system $(S_n,S)$ is the usual braid group denoted by $B_n$. It is defined as follows.

\begin{definition}\label{DefinitionClasBraidGroup}
The braid group $B_n$ is defined by a presentation with generators $\tilde{s}_1, \tilde{s}_2, \cdots, \tilde{s}_{n-1}$ and relations:
\begin{enumerate}
\item $\tilde{s}_{i}\tilde{s}_{i+1}\tilde{s}_i = \tilde{s}_{i+1}\tilde{s}_i\tilde{s}_{i+1}$ for $1 \leq i \leq n-2$,
\item $\tilde{s}_i\tilde{s}_j=\tilde{s}_j\tilde{s}_i$ for $|i-j| > 1$.
\end{enumerate}
\end{definition}

This presentation can be described by the following diagram. The nodes are the generators of the presentation. An edge between two adjacent nodes describes Relation 1 of Definition \ref{DefinitionClasBraidGroup}. Relation 2 is described by the fact that there is no edge between the corresponding nodes.

\begin{figure}[H]\label{PresofB_nDiagram}

\begin{center}
\begin{tikzpicture}

\node[draw, shape=circle, label=above:$\tilde{s}_1$] (1) at (0,0) {};
\node[draw, shape=circle, label=above:$\tilde{s}_2$] (2) at (1,0) {};
\node[draw, shape=circle,label=above:$\tilde{s}_{n-2}$] (n-2) at (4,0) {};
\node[draw,shape=circle,label=above:$\tilde{s}_{n-1}$] (n-1) at (5,0) {};

\draw[thick,-] (1) to (2);
\draw[thick,dashed,-] (2) to (n-2);
\draw[thick,-] (n-2) to (n-1);

\end{tikzpicture}
\end{center}

\caption{Diagram for the presentation of $B_n$.}
\end{figure}
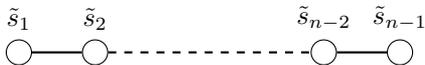

It is widely believed that complex braid groups share similar properties with Artin-Tits groups. One would like to extend what is known for Artin-Tits groups to other complex braid groups. An important feature about finite-type Artin-Tits groups is that they admit Garside structures. One of the important aspect of a Garside structure is the existence of normal forms that enable us to solve the Word and Conjugacy problems for the associated groups.  Garside structures also enjoy important group-theoretical, homological and homotopical properties. It is therefore interesting to construct Garside structures for a given group. Subsections 1.1 and 1.2 are devoted to introduce the concept of Garside structures and interval Garside structures.\\

For instance, it is shown by Bessis and Corran \cite{BessisCorranNonCrossingPartitions} in 2004, and by Corran and Picantin \cite{CorranPicantin} in 2009 that the complex braid group $B(e,e,n)$ admits Garside structures. The aim of this paper is to construct interval Garside structures for $B(e,e,n)$ that derive from natural and explicit intervals in the associated complex reflection group.\\

In Section \ref{SectionLength}, we determine geodesic normal forms for all the elements of $G(e,e,n)$ over an appropriate generating set and present an algorithm to explicitly compute them.  As an application, we are able to determine the elements of maximal length in $G(e,e,n)$ and to construct intervals. Next, we prove that these intervals are lattices which gives rise to interval Garside structures, see Theorem \ref{TheoremIntervalStructure}. The combinatorial techniques used are elementary and the associated characterizations are elegant. In Section \ref{SectionIsomB(een)}, we provide nice presentations for these structures that we denote by $B^{(k)}(e,e,n)$ for $1 \leq k \leq e-1$ and identify which of them are isomorphic to $B(e,e,n)$. For the Garside structures that are not isomorphic to $B(e,e,n)$, we provide some of their properties and compute their second integral homology groups, see Proposition \ref{PropH2ofBkeen}. Thus, we get a complete classification (in Theorem \ref{TheoremBkIsomBiff}) of the interval Garside structures of the complex braid groups $B(e,e,n)$ that derive from our construction:
\begin{center}
$B^{(k)}(e,e,n)$ is isomorphic to $B(e,e,n)$ if and only if $k$ and $e$ are coprime.\\
\end{center}

These Garside structures have been implemented by Michel and the author by using the development version of the CHEVIE package for GAP3 (see \cite{CHEVIEJMichel} and \cite{GAP3CPMichelNeaime}). The explicitness of the geodesic normal forms is what makes this implementation possible. In the next part of the introduction, we include the necessary preliminaries to accurately describe these Garside structures.

\subsection{Garside monoids and groups}

In his PhD thesis, defended in 1965 \cite{GarsideThesis}, and in the article that followed \cite{GarsideArt}, Garside solved the Conjugacy Problem for the braid group $B_n$ by introducing a submonoid $B_{n}^{+}$ of $B_n$ and an element $\Delta_n$ of $B_{n}^{+}$ that he called fundamental, and then showing that there exists a normal form for every element in $B_n$. In the beginning of the 1970's, it was realized by Brieskorn and Saito \cite{GarsideBrieskornSaito} and Deligne \cite{GarsideDeligne} that Garside's results extend to all finite-type Artin-Tits groups. At the end of the 1990's, after listing the abstract properties of $B_n^{+}$ and the fundamental element $\Delta_n$, Dehornoy and Paris \cite{GarsideDehornoyParis} defined the notion of Gaussian groups and Garside groups which leads, in ``a natural, but slowly emerging program'' as stated in \cite{GarsideBookPatrick}, to Garside theory. For a detailed study about Garside structures, we refer the reader to \cite{GarsideBookPatrick}.\\

Let $M$ be a monoid. Under some assumptions about $M$, more precisely the assumptions 1 and 2 of Definition \ref{DefGarsideMonoid} below, one can define a partial order relation on $M$ as follows.

\begin{definition}\label{DefinitionLeftDivisorMonoid}

Let $f,g \in M$. We say that $f$ left-divides $g$ or simply $f$ divides $g$ when there is no confusion, written $f \preceq g$, if $fg'=g$ holds for some $g' \in M$. Similarly, we say that $f$ right-divides $g$, written $f \preceq_r g$, if $g'f=g$ holds for some $g' \in M$. 

\end{definition}

The definition of a Garside monoid is the following.

\begin{definition}\label{DefGarsideMonoid}

A Garside monoid is a pair $(M, \Delta)$, where $M$ is a monoid and

\begin{enumerate}

\item $M$ is cancellative, that is $fg=fh \Longrightarrow g=h$ and $gf=hf \Longrightarrow g=h$ for $f,g,h \in M$,

\item there exists $\lambda : M \longrightarrow \mathbb{N}$ s.t. $\lambda(fg) \geq \lambda(f) + \lambda(g)$ and $g \neq 1 \Longrightarrow \lambda(g) \neq 0$,

\item any two elements of $M$ have a gcd and an lcm for $\preceq$ and $\preceq_r$,

\item $\Delta$ is a \emph{Garside element} of $M$, this meaning that the set of its left divisors coincides with the set of its right divisors, generate $M$, and is finite. 

\end{enumerate} 

The divisors of $\Delta$ are called the \emph{simples} of $M$.

\end{definition}

A quasi-Garside monoid is a pair $(M,\Delta)$ that satisfies the conditions of Definition \ref{DefGarsideMonoid}, except the finiteness of the number of divisors of $\Delta$.\\

Assumptions 1 and 3 of Definition \ref{DefGarsideMonoid} ensure that Ore's conditions (see \cite{GarsideBookPatrick}) are satisfied. Hence there exists a group of fractions of the monoid $M$ in which it embeds. This allows us to give the following definition.

\begin{definition}

A Garside group is the group of fractions of a Garside monoid.

\end{definition}

Consider a pair $(M,\Delta)$ that satisfies the conditions of Definition \ref{DefGarsideMonoid}. The pair $(M,\Delta)$ and the group of fractions of $M$ provide a Garside structure for $M$.

\subsection{Interval Garside structures}

Let $G$ be a finite group generated by a finite set $S$. There is a way to construct Garside structures from intervals in $G$. Let us start by defining a partial order relation on $G$.

\begin{definition}\label{DefPartalOrder11}

Let $f, g \in G$. We say that $g$ is a divisor of $f$ or $f$ is a multiple of $g$, and write $g \preceq f$, if $f = g h$ with $h \in G$ and $\ell(f) = \ell(g) + \ell(h)$, where $\ell(f)$ is the length over $S$ of $f \in G$.

\end{definition}

\begin{definition}\label{DefGeneralIntervalMonoid}

For $w \in G$, define a monoid $M([1,w])$ by the monoid presentation with

\begin{itemize}

\item generating set $\underline{P}$ in bijection with the interval
\begin{center} $[1,w] := \{ f \in G \ | \ 1 \preceq f \preceq w \}$ and \end{center}

\item relations: $\underline{f}\ \underline{g} = \underline{h}$ if $f, g , h \in [1,w]$, $fg=h$, and $f \preceq h$, that is \mbox{$\ell(f) + \ell(g) = \ell(h)$}.

\end{itemize}

\end{definition}

\noindent Similarly, one can define the partial order relation on $G$ as follows: 
\begin{center}$g \preceq_r f$ if and only if $\ell(f{g}^{-1}) + \ell(g) = \ell(f)$, \end{center}
then define the interval $[1,w]_r$ and the monoid $M([1,w]_r)$. 

\begin{definition}

Let $w$ be in $G$. We say that $w$ is a balanced element of $G$ if $[1,w] = [1,w]_r$.

\end{definition}

We have the following theorem due to Michel (see Section 10 of \cite{CorseJeanMichel} for a proof).

\begin{theorem}\label{TheoremMichelGarside}

If $w \in G$ is balanced and both posets $([1,w],\preceq)$ and $([1,w]_r, \preceq_r)$ are lattices, then $(M([1,w]),\underline{w})$ is a Garside monoid with simples $\underline{[1,w]}$, where $\underline{w}$ and $\underline{[1,w]}$ are given in Definition \ref{DefGeneralIntervalMonoid}.

\end{theorem}

The monoid $M([1,w])$ is called an interval Garside monoid. Since it is a Garside monoid, its group of fractions exists and is denoted by $G(M([1,w]))$. We call it an interval Garside group. An interval Garside structure consists of an interval Garside monoid and its group of fractions. We will give a classical example of this structure. It shows that Artin-Tits groups admit interval Garside structures. For proofs and details about the statements given in this example, see Chapter 9, Section 1.3 in \cite{GarsideBookPatrick}.

\begin{example}\label{ExampleArtinTitsIntervals}

Let $W$ be a finite Coxeter group and $B(W)$ the Artin-Tits group associated with $W$.

\begin{center} $W = <S\ |\ s^2=1, \ \underset{m_{st}}{\underbrace{sts\cdots}}=\underset{m_{st}}{\underbrace{tst\cdots}}$ for $s, t \in S, s \neq t, m_{st} = o(st)>,$\\
$B(W) = <\widetilde{S}\ |\ \underset{m_{st}}{\underbrace{\tilde{s}\tilde{t}\tilde{s}\cdots}}=\underset{m_{st}}{\underbrace{\tilde{t}\tilde{s}\tilde{t}\cdots}}\ for\ \tilde{s}, \tilde{t} \in \widetilde{S},\ \tilde{s} \neq \tilde{t}>$.
\end{center}
Take $G=W$ and $g=w_0$ the longest element over $S$ in $W$. We have \mbox{$[1,w_0] = W$}. Construct the interval monoid $M([1,w_0])$ as in Definition \ref{DefGeneralIntervalMonoid}. We have $M([1,w_0])$ is the Artin-Tits monoid $B^{+}(W)$, where $B^{+}(W)$ is the monoid defined by the same presentation as $B(W)$, see Definition \ref{DefinitionArtinTits}. Hence $B^{+}(W)$ is generated by a copy $\underline{W}$ of $W$ with $\underline{f}\ \underline{g} = \underline{h}$ if $fg=h$ and $\ell(f) + \ell(g) = \ell(h)$; $f,g$, and $h \in W$. It is also known that $w_0$ is balanced and both posets $([1,w_0],\preceq)$ and $([1,w_0]_r, \preceq_r)$ are lattices. Hence by Theorem \ref{TheoremMichelGarside}, we have the following result.

\end{example}

\begin{theorem}

$(B^{+}(W),\underline{w_0})$ is a Garside monoid with simples $\underline{W}$, where $\underline{w_0}$ and $\underline{W}$ are given in Example \ref{ExampleArtinTitsIntervals}.

\end{theorem}

\section{Presentations for \texorpdfstring{$G(e,e,n)$}{TEXT} and \texorpdfstring{$B(e,e,n)$}{TEXT}}\label{SectionGeenBeen}

In this section, we recall the presentations by generators and relations of $G(e,e,n)$ and $B(e,e,n)$ according to the results of \cite{BMR} and \cite{CorranPicantin}.

\subsection{Presentations of Brou\'e, Malle and Rouquier}

Recall that for $e,n \geq 1$, $G(e,e,n)$ is the group of $n \times n$ matrices consisting of monomial matrices, with all nonzero entries lying in $\mu_{e}$, the $e$-th roots of unity, and for which the product of the nonzero entries is $1$. Note that this family includes three families of finite Coxeter groups: $G(1,1,n)$ is the symmetric group, $G(e,e,2)$ is the dihedral group, and $G(2,2,n)$ corresponds to Coxeter-type $D_n$. We have the following result, see \cite{BMR}.

\begin{proposition}\label{PropPresBMGG(een)}
The complex reflection group $G(e,e,n)$ is isomorphic to the group defined by a presentation with generators $\mathbf{t}_0$, $\mathbf{t}_1$, $\mathbf{s}_3$, $\mathbf{s}_4$, $\cdots$, $\mathbf{s}_{n-1}$, $\mathbf{s}_n$ and relations as follows.
\begin{enumerate}

\item quadratic relations for all the generators,
\item the braid relations for $\mathbf{s}_3$, $\mathbf{s}_4$, $\cdots$, $\mathbf{s}_{n-1}$ given earlier in Definition \ref{DefinitionClasBraidGroup},
\item $\mathbf{s}_3 \mathbf{t}_i \mathbf{s}_3 = \mathbf{t}_i \mathbf{s}_3 \mathbf{t}_i$ for $i=0,1$,
\item $\mathbf{s}_j \mathbf{t}_i = \mathbf{t}_i \mathbf{s}_j$ for $i=0,1$ and $4 \leq j \leq n$,
\item $\mathbf{s}_3 \mathbf{t}_1\mathbf{t}_0\mathbf{s}_3\mathbf{t}_1\mathbf{t}_0 = \mathbf{t}_1 \mathbf{t}_0\mathbf{s}_3\mathbf{t}_1\mathbf{t}_0\mathbf{s}_3$,
\item $\underset{e}{\underbrace{\mathbf{t}_0\mathbf{t}_1\mathbf{t}_0\cdots}} = \underset{e}{\underbrace{\mathbf{t}_1\mathbf{t}_0\mathbf{t}_1\cdots}}$.

\end{enumerate}

\end{proposition}

This presentation is called the presentation of BMR (Brou\'e-Malle-Rouquier) of $G(e,e,n)$. It can be described by the following diagram.

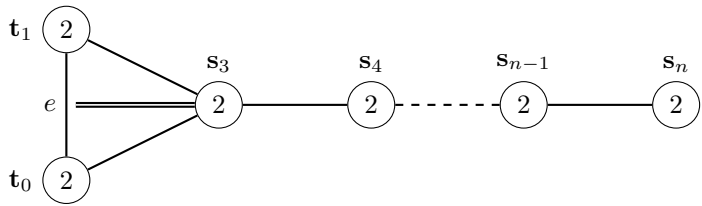
\begin{figure}[H]\label{PresofBMRofGeen}

\begin{center}
\begin{tikzpicture}

\node[draw, shape=circle, label=left:$\mathbf{t}_0$] (1) at (0,0) {2};
\node[draw, shape=circle, label=left:$\mathbf{t}_1$] (2) at (0,2) {2};
\node[draw, shape=circle,label=above:$\mathbf{s}_3$] (3) at (2,1) {2};
\node[draw, shape=circle,label=above:$\mathbf{s}_4$] (4) at (4,1) {2};
\node[draw, shape=circle,label=above:$\mathbf{s}_{n-1}$] (n-1) at (6,1) {2};
\node[draw,shape=circle,label=above:$\mathbf{s}_n$] (n) at (8,1) {2};

\node[] (0) at (0,1) {};

\draw[thick,-] (1) to node[auto] {$e$} (2);
\draw[thick,-] (1) to (3);
\draw[thick,-] (2) to (3);
\draw[thick,-] (3) to (4);
\draw[thick,dashed,-] (4) to (n-1);
\draw[thick,-] (n-1) to (n);
\draw[double,thick,-] (3) to (0);

\end{tikzpicture}
\end{center}

\caption{Diagram for the presentation of BMR of $G(e,e,n)$.}
\end{figure}

The matrices in $G(e,e,n)$ that correspond to the generators are given by $\mathbf{t}_i \longmapsto t_i:= \begin{pmatrix}

0 & \zeta_{e}^{-i} & 0\\
\zeta_{e}^{i} & 0 & 0\\
0 & 0 & I_{n-2}\\

\end{pmatrix}$ for $i=0,1$ and $\mathbf{s}_j \longmapsto s_j:= \begin{pmatrix}

I_{j-2} & 0 & 0 & 0\\
0 & 0 & 1 & 0\\
0 & 1 & 0 & 0\\
0 & 0 & 0 & I_{n-j}\\

\end{pmatrix}$ for $3 \leq j \leq n$, where $\zeta_e$ is the $e$-th root of unity that is equal to $exp(2i\pi/e)$ and $I_k$ is the identity matrix for $1 \leq k \leq n$.\\

According to \cite{BMR}, if we remove the quadratic relations in Proposition \ref{PropPresBMGG(een)}, we get a presentation of the complex braid group $B(e,e,n)$ associated with $G(e,e,n)$ that we call the presentation of BMR of $B(e,e,n)$. The generators of this presentation are in bijection with $\mathbf{t}_0$, $\mathbf{t}_1$, $\mathbf{s}_3$, $\mathbf{s}_4$, $\cdots$, $\mathbf{s}_n$ and are denoted by $\tilde{t}_0$, $\tilde{t}_1$, $\tilde{s}_3$, $\tilde{s}_4$, $\cdots$, $\tilde{s}_n$. The diagram of this presentation is the following.

\begin{figure}[H]\label{PresofBMRBeen}
\begin{center}
\begin{tikzpicture}

\node[draw, shape=circle, label=left:$\tilde{t}_0$] (1) at (0,0) {};
\node[draw, shape=circle, label=left:$\tilde{t}_1$] (2) at (0,2) {};
\node[draw, shape=circle,label=above:$\tilde{s}_3$] (3) at (2,1) {};
\node[draw, shape=circle,label=above:$\tilde{s}_4$] (4) at (4,1) {};
\node[draw, shape=circle,label=above:$\tilde{s}_{n-1}$] (n-1) at (6,1) {};
\node[draw,shape=circle,label=above:$\tilde{s}_n$] (n) at (8,1) {};

\node[] (0) at (0,1) {};

\draw[thick,-] (1) to node[auto] {$e$} (2);
\draw[thick,-] (1) to (3);
\draw[thick,-] (2) to (3);
\draw[thick,-] (3) to (4);
\draw[thick,dashed,-] (4) to (n-1);
\draw[thick,-] (n-1) to (n);
\draw[double,thick,-] (3) to (0);

\end{tikzpicture}
\end{center}

\caption{Diagram for the presentation of BMR of $B(e,e,n)$.}
\end{figure}
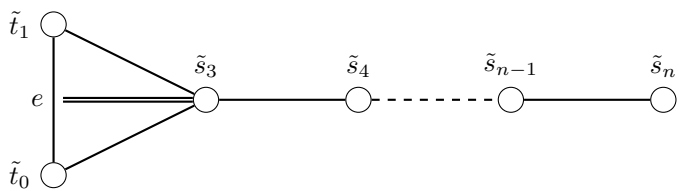

\subsection{Presentations of Corran and Picantin}\label{SubsectionOtherPresentations}

Consider the presentation of BMR of the complex braid group $B(e,e,n)$. For $e \geq 3$ and $n \geq 3$, it is shown in \cite{CorranPhD} that the monoid defined by the corresponding monoid presentation fails to embed in $B(e,e,n)$. Thus, this presentation does not give rise to a Garside structure for $B(e,e,n)$. In \cite{CorranPicantin}, Corran and Picantin described another presentation for $B(e,e,n)$ and showed that it gives rise to a Garside structure for $B(e,e,n)$. We have the following.

\begin{proposition}\label{PropositionPresCPB(een)}
The complex braid group $B(e,e,n)$ is isomorphic to a group defined by a presentation with generators $\tilde{t}_0$, $\tilde{t}_1$, $\cdots$, $\tilde{t}_{e-1}$, $\tilde{s}_3$, $\tilde{s}_4$, $\cdots$, $\tilde{s}_{n-1}$, $\tilde{s}_n$ and relations as follows.
\begin{enumerate}

\item the braid relations for $\tilde{s}_3$, $\tilde{s}_4$, $\cdots$, $\tilde{s}_{n-1}$ given earlier in Definition \ref{DefinitionClasBraidGroup},
\item $\tilde{s}_3 \tilde{t}_i \tilde{s}_3 = \tilde{t}_i \tilde{s}_3 \tilde{t}_i$ for $0 \leq i \leq e-1$,
\item $\tilde{s}_j \tilde{t}_i = \tilde{t}_i \tilde{s}_j$ for $0 \leq i \leq e-1$ and $4 \leq j \leq n$,
\item $\tilde{t}_i \tilde{t}_{i-1} = \tilde{t}_j \tilde{t}_{j-1}$ for $i,j \in \mathbb{Z}/e\mathbb{Z}$.

\end{enumerate}

\end{proposition}

This presentation is called the presentation of Corran-Picantin of $B(e,e,n)$. It can be described by the following diagram (the kite). The dashed circle describes Relations 4 of Proposition \ref{PropositionPresCPB(een)}. The other edges used to describe all the other relations follow the standard conventions for the usual braid group diagram, see Figure \ref{PresofB_nDiagram}.

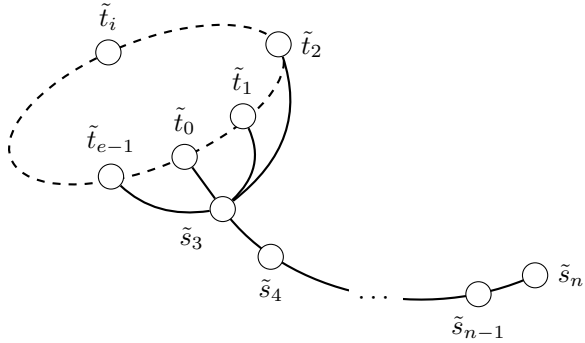
\begin{figure}[H]
\begin{center}
\begin{tikzpicture}[yscale=0.8,xscale=1,rotate=30]

\draw[thick,dashed] (0,0) ellipse (2cm and 1cm);

\node[draw, shape=circle, fill=white, label=above:$\tilde{t}_0$] (t0) at (0,-1) {};
\node[draw, shape=circle, fill=white, label=above:$\tilde{t}_1$] (t1) at (1,-0.8) {};
\node[draw, shape=circle, fill=white, label=right:$\tilde{t}_2$] (t2) at (2,0) {};
\node[draw, shape=circle, fill=white, label=above:$\tilde{t}_i$] (ti) at (0,1) {};
\node[draw, shape=circle, fill=white, label=above:$\tilde{t}_{e-1}$] (te-1) at (-1,-0.8) {};

\draw[thick,-] (0,-2) arc (-180:-90:3);

\node[draw, shape=circle, fill=white, label=below left:$\tilde{s}_3$] (s3) at (0,-2) {};

\draw[thick,-] (t0) to (s3);
\draw[thick,-,bend left] (t1) to (s3);
\draw[thick,-,bend left] (t2) to (s3);
\draw[thick,-,bend left] (s3) to (te-1);

\node[draw, shape=circle, fill=white, label=below:$\tilde{s}_4$] (s4) at (0.15,-3) {};
\node[draw, shape=circle, fill=white, label=below:$\tilde{s}_{n-1}$] (sn-1) at (2.2,-4.9) {};
\node[draw, shape=circle, fill=white, label=right:$\tilde{s}_{n}$] (sn) at (3,-5) {};

\node[fill=white] () at (1,-4.285) {$\cdots$};

\end{tikzpicture}
\end{center}

\caption{Diagram for the presentation of Corran-Picantin of $B(e,e,n)$.}\label{PresofCPBeen}
\end{figure}

One of the important results proved in \cite{CorranPicantin} is the following.

\begin{theorem}\label{TheoremCPGarside}

The presentation of Corran-Picantin gives rise to a Garside structure for $B(e,e,n)$ with

\begin{itemize}
\item Garside element: $\Delta = \underset{\Delta_2}{\underbrace{\tilde{t}_1\tilde{t}_0}}\ \underset{\Delta_3}{\underbrace{\tilde{s}_3\tilde{t}_1\tilde{t}_0\tilde{s}_3}}\ \cdots\ \underset{\Delta_n}{\underbrace{\tilde{s}_n\tilde{s}_{n-1} \cdots \tilde{s}_3\tilde{t}_1\tilde{t}_0\tilde{s}_3 \cdots \tilde{s}_n}}$,

\item Simples: the elements of the form $\delta_2 \delta_3 \cdots \delta_n$ where $\delta_i$ is a left-divisor of $\Delta_i$ (see Definition \ref{DefinitionLeftDivisorMonoid}) for $2 \leq i \leq n$.
\end{itemize}

\end{theorem}

It is also shown in \cite{CorranPicantin} that if one adds the quadratic relations for all the generators of the presentation of Corran-Picantin of $B(e,e,n)$, one obtains a presentation of a group isomorphic to $G(e,e,n)$. It is called the presentation of Corran-Picantin of $G(e,e,n)$. Its diagram is as follows.

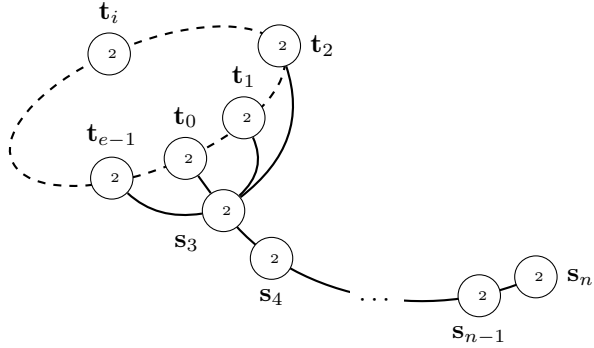
\begin{figure}[H]\label{PresofCPofGeen}
\begin{center}
\begin{tikzpicture}[yscale=0.8,xscale=1,rotate=30]

\draw[thick,dashed] (0,0) ellipse (2cm and 1cm);

\node[draw, shape=circle, fill=white, label=above:$\mathbf{t}_0$] (t0) at (0,-1) {\begin{tiny} 2 \end{tiny}};
\node[draw, shape=circle, fill=white, label=above:$\mathbf{t}_1$] (t1) at (1,-0.8) {\begin{tiny} 2 \end{tiny}};
\node[draw, shape=circle, fill=white, label=right:$\mathbf{t}_2$] (t2) at (2,0) {\begin{tiny} 2 \end{tiny}};
\node[draw, shape=circle, fill=white, label=above:$\mathbf{t}_i$] (ti) at (0,1) {\begin{tiny} 2 \end{tiny}};
\node[draw, shape=circle, fill=white, label=above:$\mathbf{t}_{e-1}$] (te-1) at (-1,-0.8) {\begin{tiny} 2 \end{tiny}};

\draw[thick,-] (0,-2) arc (-180:-90:3);

\node[draw, shape=circle, fill=white, label=below left:$\mathbf{s}_3$] (s3) at (0,-2) {\begin{tiny} 2 \end{tiny}};

\draw[thick,-] (t0) to (s3);
\draw[thick,-,bend left] (t1) to (s3);
\draw[thick,-,bend left] (t2) to (s3);
\draw[thick,-,bend left] (s3) to (te-1);

\node[draw, shape=circle, fill=white, label=below:$\mathbf{s}_4$] (s4) at (0.15,-3) {\begin{tiny} 2 \end{tiny}};
\node[draw, shape=circle, fill=white, label=below:$\mathbf{s}_{n-1}$] (sn-1) at (2.2,-4.9) {\begin{tiny} 2 \end{tiny}};
\node[draw, shape=circle, fill=white, label=right:$\mathbf{s}_{n}$] (sn) at (3,-5) {\begin{tiny} 2 \end{tiny}};

\node[fill=white] () at (1,-4.285) {$\cdots$};

\end{tikzpicture}
\end{center}

\caption{Diagram for the presentation of Corran-Picantin of $G(e,e,n)$.}
\end{figure}

\begin{remark}\

\begin{enumerate}

\item \mbox{For $e=1$ and $n \geq 2$, we get the classical presentation of the symmetric group $S_n$.}
\item For $e=2$ and $n \geq 2$, we get the classical presentation of the Coxeter group of type $D_n$.

\end{enumerate}

\end{remark}

Denote by $\mathbf{X}$ the set $\{ \mathbf{t}_0, \mathbf{t}_1, \cdots, \mathbf{t}_{e-1}, \mathbf{s}_3, \cdots, \mathbf{s}_n \}$ of the generators of the presentation of Corran-Picantin of $G(e,e,n)$. The matrices in $G(e,e,n)$ that correspond to the generators are given by $\mathbf{t}_i \longmapsto t_i:= \begin{pmatrix}

0 & \zeta_{e}^{-i} & 0\\
\zeta_{e}^{i} & 0 & 0\\
0 & 0 & I_{n-2}\\

\end{pmatrix}$ \mbox{for $0 \leq i \leq e-1$} and $\mathbf{s}_j \longmapsto s_j:= \begin{pmatrix}

I_{j-2} & 0 & 0 & 0\\
0 & 0 & 1 & 0\\
0 & 1 & 0 & 0\\
0 & 0 & 0 & I_{n-j}\\

\end{pmatrix}$ for $3 \leq j \leq n$, where $\zeta_e$ is the $e$-th root of unity that is equal to $exp(2i\pi/e)$ and $I_k$ is the identity matrix for $1 \leq k \leq n$. Denote by $X$ the set $\{t_0,t_1, \cdots, t_{e-1},s_3, \cdots, s_n\}$.\\

To avoid confusion, we use regular letters for matrices in $G(e,e,n)$ and bold letters for words over $\mathbf{X}$. We also set the following convention.

\begin{convention}\label{ConventionDecreaseIncreaseIndex}

A decreasing-index expression of the form $\mathbf{s}_i \mathbf{s}_{i-1} \cdots \mathbf{s}_{i'}$ is the empty word when $i < i'$ and an increasing-index expression of the form $\mathbf{s}_i \mathbf{s}_{i+1} \cdots \mathbf{s}_{i'}$ is the empty word when $i > i'$. Similarly, in $G(e,e,n)$, a decreasing-index product of the form $s_i s_{i-1} \cdots s_{i'}$ is equal to $I_n$ when $i < i'$ and an increasing-index product of the form $s_i s_{i+1} \cdots s_{i'}$ is equal to $I_n$ when $i > i'$, where $I_n$ is the identity $n \times n$ matrix.

\end{convention}

\section{Geodesic normal forms for \texorpdfstring{$G(e,e,n)$}{TEXT}}\label{SectionLength}

In this section, we define geodesic normal forms for the elements of the general series of complex reflection groups $G(e,e,n)$ by using the presentation of Corran-Picantin of these groups. Actually, we define an algorithm that produces a word representative for each element of $G(e,e,n)$ over $\mathbf{X}$, where $\mathbf{X}$ is the set of the generators of the presentation of Corran-Picantin. Then we prove that these word representatives are geodesic. Hence we get geodesic normal forms for $G(e,e,n)$. As an application, we determine the elements of $G(e,e,n)$ that are of maximal length over $\mathbf{X}$.

\subsection{Minimal word representatives}

Denote by $\boldsymbol{\ell}(\boldsymbol{w})$ the word length over $\mathbf{X}$ of the word $\boldsymbol{w} \in \mathbf{X}^{*}$, where $\mathbf{X}$ is the set of the generators of the presentation of Corran-Picantin. Let us start by the following definition.

\begin{definition}\label{def.Slength.reducedword}
Let $w$ be an element of $G(e,e,n)$.
We define $\ell(w)$ to be the minimal word length $\boldsymbol{\ell}(\boldsymbol{w})$ of a word $\boldsymbol{w}$ over $\mathbf{X}$ that represents $w$. A reduced expression of $w$ is any word representative of $w$ of word length $\ell(w)$.
\end{definition}

Our aim is to represent each element of $G(e,e,n)$ by a reduced word over $\mathbf{X}$. This requires the elaboration of a combinatorial technique to determine a reduced expression decomposition over $\mathbf{X}$ for an element of $G(e,e,n)$.\\

We introduce the algorithm below (see next page) that produces a word $R\!E(w)$ over $\mathbf{X}$ for a given matrix $w$ in $G(e,e,n)$. Note that we use Convention \ref{ConventionDecreaseIncreaseIndex} in the elaboration of the algorithm. Later on, we prove that $R\!E(w)$ is a reduced expression over $\mathbf{X}$ of $w$, see Proposition \ref{PropREwRedExp}.\\

Let $w_n := w \in G(e,e,n)$. For $i$ from $n$ to $2$, the $i$-th step of the algorithm transforms the block diagonal matrix $\left(
\begin{array}{c|c}
w_i & 0 \\
\hline
0 & I_{n-i}
\end{array}
\right)$ into a block diagonal matrix $\left(
\begin{array}{c|c}
w_{i-1} & 0 \\
\hline
0 & I_{n-i+1}
\end{array}
\right) \in G(e,e,n)$ with $w_1 = 1$. Actually, for $2 \leq i \leq n$, there exists a unique $c$ with $1 \leq c \leq n$ such that $w_i[i,c] \neq 0$. At each step $i$ of the algorithm, if $w_i[i,c] =1$, we shift it into the diagonal position $[i,i]$ by right multiplication by transpositions of the symmetric group $S_n$. If $w_i[i,c] \neq 1$, we shift it into the first column by right multiplication by transpositions, transform it into $1$ by right multiplication by an element of $\{t_0, t_1, \cdots, t_{e-1} \}$, and then shift the $1$ obtained into the diagonal position $[i,i]$.\\ 

\begin{algorithm}[]\label{Algo1}

\SetKwInOut{Input}{Input}\SetKwInOut{Output}{Output}

\noindent\rule{12cm}{0.5pt}

\Input{$w$, a matrix in $G(e,e,n)$, with $e \geq 1$ and $n \geq 2$.}
\Output{$R\!E(w)$, a word over $\mathbf{X}$.}

\noindent\rule{12cm}{0.5pt}

\textbf{Local variables}: $w'$, $R\!E(w)$,  $i$, $U$, $c$, $k$.

\noindent\rule{12cm}{0.5pt}

\textbf{Initialisation}:
$U:=[1,\zeta_e,\zeta_{e}^2,...,\zeta_{e}^{e-1}]$, $s_2 := t_0$, $\mathbf{s}_2 := \mathbf{t}_0$,\\
$R\!E(w) := \varepsilon$: the empty word, $w' := w$.

\noindent\rule{12cm}{0.5pt}

\For{$i$ \textbf{from} $n$ \textbf{down to} $2$} {
	$c:=1$; $k:=0$; \\
	\While{$w'[i,c] = 0$}{$c :=c+1$\; 
	}
	 \textit{\#Then $w'[i,c]$ is the root of unity on the row $i$}\;
	\While{$U[k+1]\neq w'[i,c] $}{$k :=k+1$\;
	}
	\textit{\#Then $w'[i,c] = \zeta_{e}^k$.}\\

		\If{$k \neq 0$}{
		$w' := w's_{c}s_{c-1} \cdots s_{3}s_{2}t_{k}$; \textit{\#Then $w'[i,2] =1$}\;
		$R\!E(w) := \mathbf{t}_k\mathbf{s}_2\mathbf{s}_3 \cdots \mathbf{s}_c R\!E(w)$\;
		$c:=2$\;
	}	
	$w' := w's_{c+1} \cdots s_{i-1} s_{i}$; \textit{\#Then $w'[i,i] = 1$}\;
	$R\!E(w) := \mathbf{s}_i \mathbf{s}_{i-1} \cdots \mathbf{s}_{c+1} R\!E(w)$\;
}
\textbf{Return} $R\!E(w)$;

\noindent\rule{12cm}{0.5pt}

\caption{A word over $\mathbf{X}$ corresponding to an element $w \in G(e,e,n)$.}
\end{algorithm}

We provide two examples in order to better understand the algorithm. The first one is for an element $w$ of $G(3,3,4)$ and the second example is for an element $w$ of $G(2,2,4)$, that is the Coxeter group of type $D_4$. At each step, we indicate the values of $i$, $k$, and $c$ such that $w_i[i,c] = \zeta_e^k$.

\begin{example}\label{examp algo NF}

We apply the algorithm to $w := \begin{pmatrix}

0 & 0 & 0 & 1\\
0 & \zeta_3^2 & 0 & 0\\
0 & 0 & \zeta_3 & 0\\
1 & 0 & 0 & 0\\

\end{pmatrix}$ $\in G(3,3,4)$.\\
Step $1$ $(i=4, k=0, c=1)$: $w':=ws_2s_3s_4=\begin{pmatrix}

0 & 0 & 1 & 0\\
\zeta_3^2 & 0 & 0 & 0\\
0 & \boxed{\zeta_3} & 0 & 0\\
0 & 0 & 0 & \mathbf{1}\\

\end{pmatrix}$.\\
Step $2$ $(i=3, k=1, c=2)$: $w' := w's_2 = \begin{pmatrix}

0 & 0 & 1 & 0\\
0 & \zeta_3^2 & 0 & 0\\
\zeta_3 & 0 & 0 & 0\\
0 & 0 & 0 & 1\\

\end{pmatrix}$,\\
then $w' := w't_1 = \begin{pmatrix}

0 & 0 & 1 & 0\\
1 & 0 & 0 & 0\\
0 & 1 & 0 & 0\\
0 & 0 & 0 & 1\\

\end{pmatrix}$, then $w':=w's_3 = \begin{pmatrix}

0 & 1 & 0 & 0\\
\boxed{1} & 0 & 0 & 0\\
0 & 0 & \mathbf{1} & 0\\
0 & 0 & 0 & 1\\

\end{pmatrix}$.\\
Step $3$ $(i=2, k=0, c=1)$: $w' := w's_2 = I_{4}$.\\
Hence $R\!E(w) = \mathbf{s}_2 \mathbf{s}_3 \mathbf{t}_1 \mathbf{s}_2 \mathbf{s}_4 \mathbf{s}_3 \mathbf{s}_2$. Recall that $\mathbf{s_2} = \mathbf{t}_0$. Thus, $R\!E(w) = \mathbf{t}_0 \mathbf{s}_3 \mathbf{t}_1 \mathbf{t}_0 \mathbf{s}_4 \mathbf{s}_3 \mathbf{t}_0$.\\

\end{example}

\begin{example}

We apply the algorithm to $w := \begin{pmatrix}

0 & 1 & 0 & 0\\
0 & 0 & 0 & -1\\
0 & 0 & 1 & 0\\
-1 & 0 & 0 & 0\\

\end{pmatrix}$ $\in G(2,2,4)$.\\
Step $1$ $(i=4, k=1, c=1)$: $w':=wt_1=\begin{pmatrix}

-1 & 0 & 0 & 0\\
0 & 0 & 0 & -1\\
0 & 0 & 1 & 0\\
0 & 1 & 0 & 0\\

\end{pmatrix}$,\\ then $w':=w' s_3s_4=\begin{pmatrix}

-1 & 0 & 0 & 0\\
0 & 0 & -1 & 0\\
0 & 1 & 0 & 0\\
0 & 0 & 0 & 1\\

\end{pmatrix}$,\\
Step $2$ $(i=3, k=0, c=2)$: $w' := w's_3 = \begin{pmatrix}

-1 & 0 & 0 & 0\\
0 & -1 & 0 & 0\\
0 & 0 & 1 & 0\\
0 & 0 & 0 & 1\\

\end{pmatrix}$.\\
Step $3$ $(i=2, k=1, c=2)$: $w' := w's_2 = \begin{pmatrix}

0 & -1 & 0 & 0\\
-1 & 0 & 0 & 0\\
0 & 0 & 1 & 0\\
0 & 0 & 0 & 1\\

\end{pmatrix}$,\\ then $w' := w' t_1 = I_4$.\\
Hence $R\!E(w) = \mathbf{t}_1 \mathbf{s}_2 \mathbf{s}_3 \mathbf{s}_4 \mathbf{s}_3 \mathbf{t}_1 = \mathbf{t}_1 \mathbf{t}_0 \mathbf{s}_3 \mathbf{s}_4 \mathbf{s}_3 \mathbf{t}_1$ \emph{(since $\mathbf{s}_2 = \mathbf{t}_0$)}.\\

\end{example}

\begin{remark}

Let $w$ be an element of $G(e,e,2)$, that is the dihedral group $I_2(e)$. Denote by $\varepsilon$ the empty word. By the algorithm and by assuming Convention \ref{ConventionDecreaseIncreaseIndex}, $R\!E(w)$ belongs to $\{\varepsilon, \mathbf{t}_0, \mathbf{t}_1, \cdots, \mathbf{t}_{e-1}, \mathbf{t}_1\mathbf{t}_0, \mathbf{t}_2\mathbf{t}_0, \cdots, \mathbf{t}_{e-1}\mathbf{t}_0 \}$.

\end{remark}

The next lemma follows directly from the algorithm. 

\begin{lemma}\label{LemmaBlocks}

For $2 \leq i \leq n$, suppose $w_i[i,c] \neq 0$. The block $w_{i-1}$ is obtained by
\begin{itemize}

\item removing the row $i$ and the column $c$ from $w_i$, then by

\item multiplying the first column of the new matrix by $w_i[i,c]$.\\

\end{itemize}

\end{lemma}

\begin{example}

Let $w$ be as in Example \ref{examp algo NF}, where $n = 4$. The block $w_{3}$ is obtained by removing the row number $4$ and first column from $w_4 = w$ to obtain $\begin{pmatrix}

0 & 0 & 1\\
\zeta_3^2 & 0 & 0\\
0 & \zeta_3 & 0\\

\end{pmatrix}$, then by multiplying the first column of this matrix by $1$. The same can be said for the other block $w_2$.

\end{example}

\begin{definition}\label{DefREiw}

Let $2 \leq i \leq n$. Denote by $w_i[i,c]$ the unique nonzero entry on the row $i$ with $1 \leq c \leq i$.

\begin{itemize}

\item If $w_{i}[i,c] =1$, we define $R\!E_{i}(w)$ to be the word \\ 
$\mathbf{s}_i \mathbf{s}_{i-1} \cdots \mathbf{s}_{c+1}$ (decreasing-index expression).
\item If $w_{i}[i,c] = \zeta_{e}^{k}$ with $k \neq 0$, we define $R\!E_{i}(w)$ to be the word \\
			\begin{tabular}{ll}
			$\mathbf{s}_i \cdots \mathbf{s}_3 \mathbf{t}_k$ & if $c=1$,\\
			$\mathbf{s}_i \cdots \mathbf{s}_3 \mathbf{t}_k \mathbf{t}_0$  & if $c=2$,\\
			$\mathbf{s}_i \cdots \mathbf{s}_3 \mathbf{t}_k \mathbf{t}_0 \mathbf{s}_3 \cdots \mathbf{s}_c$ & if $c \geq 3$.
			\end{tabular}

\end{itemize}

\end{definition}

Remark that for $3 \leq i \leq n$, the word $R\!E_{i}(w)$ is either the empty word (when $w_{i}[i,i] =1$, see Convention \ref{ConventionDecreaseIncreaseIndex}) or a word that contains $\mathbf{s}_{i}$ necessarily but does not contain any of $\mathbf{s}_{i+1},\mathbf{s}_{i+2}, \cdots , \mathbf{s}_n$. Remark also that for $i = 2$, by using Convention \ref{ConventionDecreaseIncreaseIndex}, we have $R\!E_2(w) \in \{\varepsilon, \mathbf{t}_0, \mathbf{t}_1, \cdots, \mathbf{t}_{e-1}, \mathbf{t}_1\mathbf{t}_0, \cdots, \mathbf{t}_{e-1}\mathbf{t}_0\}$.

\begin{lemma}

We have $R\!E(w) = R\!E_2(w) R\!E_3(w) \cdots R\!E_n(w)$.

\end{lemma}

\begin{proof}

The output $R\!E(w)$ of the algorithm is a concatenation of the words\\ $R\!E_2(w), R\!E_3(w), \cdots$, and $R\!E_n(w)$ obtained at each step $i$ from $n$ to $2$ of the algorithm.

\end{proof}

\begin{example}

If $w$ is defined as in Example \ref{examp algo NF}, we have\\
$R\!E(w) = \underset{R\!E_{2}(w)}{\underbrace{\mathbf{t}_0}} \hspace{0.2cm} \underset{R\!E_{3}(w)}{\underbrace{\mathbf{s}_3\mathbf{t}_1\mathbf{t}_0}} \hspace{0.2cm} \underset{R\!E_{4}(w)}{\underbrace{ \mathbf{s}_4\mathbf{s}_3\mathbf{t}_0}}$.

\end{example}

\begin{proposition}

The word $R\!E(w)$ given by the algorithm is a word representative over $\mathbf{X}$ of $w \in G(e,e,n)$.

\end{proposition}

\begin{proof}

The algorithm transforms the matrix $w$ into $I_n$ by multiplying it on the right by elements of $X$. We get $wx_1 \cdots x_r = I_n$, where $x_1$, $\cdots$, $x_r$ are elements of $X$. Hence $w = x_r^{-1} \cdots x_1^{-1} = x_r \cdots x_1$ since $x_i^2 = 1$ for all $x_i \in X$. The output $R\!E(w)$ of the algorithm is $R\!E(w) = \mathbf{x}_r \cdots \mathbf{x}_1$. Hence it is a word representative over $\mathbf{X}$ of $w \in G(e,e,n)$.

\end{proof}

The following proposition will prepare us to prove that the output of the algorithm is a reduced expression over $\mathbf{X}$ of a given element $w \in G(e,e,n)$.

\begin{proposition}\label{prop.lengthcomp}

Let $w$ be an element of $G(e,e,n)$. For all $x \in X$, we have\\
$$|\boldsymbol{\ell}(R\!E(xw)) - \boldsymbol{\ell}(R\!E(w))| = 1.$$

\end{proposition}

\begin{proof}

For $1 \leq i \leq n$, there exists a unique $c_i$ such that $w[i,c_i] \neq 0$. We denote $w[i,c_i]$ by $a_i$.\\

\underline{Case 1: Suppose $x = s_i$ for $3 \leq i \leq n$.}\\

Set $w' := s_i w$. Since the left multiplication by the matrix $x$ exchanges the rows $i-1$ and $i$ of $w$ and the other rows remain the same, by Definition \ref{DefREiw} and Lemma \ref{LemmaBlocks}, we have:\\
$R\!E_{i+1}(xw)R\!E_{i+2}(xw) \cdots R\!E_{n}(xw)=R\!E_{i+1}(w) R\!E_{i+2}(w) \cdots R\!E_{n}(w)$ and\\
$R\!E_{2}(xw)R\!E_{3}(xw) \cdots R\!E_{i-2}(xw)=R\!E_{2}(w) R\!E_{3}(w) \cdots R\!E_{i-2}(w)$.\\
Then, in order to prove our property, we should compare $\boldsymbol{\ell}_1:=\boldsymbol{\ell}(R\!E_{i-1}(w)R\!E_i(w))$ and $\boldsymbol{\ell}_2:=\boldsymbol{\ell}(R\!E_{i-1}(xw)R\!E_i(xw))$.\\

Suppose  $c_{i-1} < c_{i}$, by Lemma \ref{LemmaBlocks}, the rows $i-1$ and $i$ of the blocks $w_i$ and $w'_{i}$ are of the form:

\begin{tabular}{ccc}

\begin{tikzpicture}[scale=0.5]

\node at (0,0) {};
\node at (0,0.8) {$w_i$ :};

\end{tikzpicture}

 & & \begin{tikzpicture}[scale=0.5]

\node at (0,1) {$i$};
\node at (0,2) {$i-1$};
\node at (2,3) {$..$};
\node at (2.6,3) {$c$};
\node at (4,3) {$..$};
\node at (5,3) {$c'$};
\node at (5.7,3) {$..$};
\node at (7,3) {$i$};

\node at (3,2) {$b_{i-1}$};
\node at (5,1) {$a_i$};

\draw [-] (1,0.5) to node[auto] {} (1,2.5);
\draw [-] (1,0.5) to node[auto] {} (7.3,0.5);
\draw [-] (7.3,0.5) to node[auto] {} (7.3,2.5);
\draw [-] (1,2.5) to node[auto] {} (7.3,2.5);

\end{tikzpicture}\\

\begin{tikzpicture}[scale=0.5]

\node at (0,0) {};
\node at (0,0.8) {$w'_i$ :};

\end{tikzpicture}

 & & \begin{tikzpicture}[scale=0.5]

\node at (0,1) {$i$};
\node at (0,2) {$i-1$};
\node at (2,3) {$..$};
\node at (2.6,3) {$c$};
\node at (4,3) {$..$};
\node at (5,3) {$c'$};
\node at (5.7,3) {$..$};
\node at (7,3) {$i$};

\node at (3,1) {$b_{i-1}$};
\node at (5,2) {$a_i$};

\draw [-] (1,0.5) to node[auto] {} (1,2.5);
\draw [-] (1,0.5) to node[auto] {} (7.3,0.5);
\draw [-] (7.3,0.5) to node[auto] {} (7.3,2.5);
\draw [-] (1,2.5) to node[auto] {} (7.3,2.5);

\end{tikzpicture}

\end{tabular}

with $c < c'$ and where we write $b_{i-1}$ instead of $a_{i-1}$ since $a_{i-1}$ may change when applying the algorithm if $c_{i-1} =1$, that is $a_{i-1}$ on the first column of $w$.\\

We will discuss different cases depending on the values of $a_i$ and $b_{i-1}$.

\begin{itemize}

\item \underline{Suppose $a_i = 1$.}

\begin{itemize}

\item \underline{If $b_{i-1} =1$,}\\
we have $R\!E_{i}(w) = \boldsymbol{s}_i \cdots \boldsymbol{s}_{c'+2} \boldsymbol{s}_{c'+1}$ and $R\!E_{i-1}(w) = \boldsymbol{s}_{i-1} \cdots \boldsymbol{s}_{c+2} \boldsymbol{s}_{c+1}$.
Furthermore, we have $R\!E_{i}(xw) = \boldsymbol{s}_i \cdots \boldsymbol{s}_{c+2} \boldsymbol{s}_{c+1}$\\
and $R\!E_{i-1}(xw) = \boldsymbol{s}_{i-1} \cdots \boldsymbol{s}_{c'+1} \boldsymbol{s}_{c'}$.\\
It follows that $\boldsymbol{\ell}_1 = ((i-1)-(c+1)+1) + (i-(c'+1)+1) = 2i-c-c'-1$ and $\boldsymbol{\ell}_2 = ((i-1)-c'+1) + (i-(c+1)+1) = 2i-c-c'$ hence $\boldsymbol{\ell}_2 = \boldsymbol{\ell}_1 +1$.\\
\item \underline{If $b_{i-1} = \zeta_e^{k}$ with $1 \leq k \leq e-1$,}\\
we have $R\!E_{i}(w) = \boldsymbol{s}_i \cdots \boldsymbol{s}_{c'+2} \boldsymbol{s}_{c'+1}$ and $R\!E_{i-1}(w) = \boldsymbol{s}_{i-1} \cdots \boldsymbol{s}_3 \boldsymbol{t}_k \boldsymbol{t}_0 \boldsymbol{s}_{3} \cdots \boldsymbol{s}_{c}$. Furthermore, we have $R\!E_{i}(xw) = \boldsymbol{s}_i \cdots \boldsymbol{s}_3 \boldsymbol{t}_k \boldsymbol{t}_0 \boldsymbol{s}_{3} \cdots \boldsymbol{s}_{c}$ and $R\!E_{i-1}(xw) = \boldsymbol{s}_{i-1} \cdots \boldsymbol{s}_{c'}$.\\
It follows that $\boldsymbol{\ell}_1 = (((i-1)-3+1) + 2 + (c-3+1)) + (i-(c'+1)+1) = 2i+c-c'-3$ and $\boldsymbol{\ell}_2 = ((i-1)-c'+1) + ((i-3+1) + 2 + (c-3+1)) = 2i+c-c'-2$ hence $\boldsymbol{\ell}_2 = \boldsymbol{\ell}_1 +1$.

\end{itemize}

It follows that

\begin{center}
if $a_i =1$, then $\boldsymbol{\ell}(R\!E(s_iw))= \boldsymbol{\ell}(R\!E(w)) +1. \hspace{1cm} (a)$
\end{center}

\item \underline{Suppose now that $a_i = \zeta_e^{k}$ with $1 \leq k \leq e-1$.}

\begin{itemize}

\item \underline{If $b_{i-1} = 1$,}\\
we have $R\!E_{i}(w) = \boldsymbol{s}_{i} \cdots \boldsymbol{s}_3 \boldsymbol{t}_k \boldsymbol{t}_0 \boldsymbol{s}_{3} \cdots \boldsymbol{s}_{c'}$ and $R\!E_{i-1}(w) =\boldsymbol{s}_{i-1} \cdots \boldsymbol{s}_{c+1}$.\\
Also, we have $R\!E_{i}(xw) = \boldsymbol{s}_i \cdots \boldsymbol{s}_{c+1}$ and\\
$R\!E_{i-1}(xw) = \boldsymbol{s}_{i-1} \cdots \boldsymbol{s}_3 \boldsymbol{t}_k \boldsymbol{t}_0 \boldsymbol{s}_{3} \cdots \boldsymbol{s}_{c'-1}$.\\
It follows that $\boldsymbol{\ell}_1 = ((i-1)-(c+1)-1) + ((i-3+1)+2+(c'-3+1)) = 2i-c+c'-5$ and $\boldsymbol{\ell}_2 = (((i-1)-3+1)+2+((c'-1)-3+1)) + (i-(c+1)-1) = 2i-c+c'-6$ hence $\boldsymbol{\ell}_2 = \boldsymbol{\ell}_1 -1$.\\

\item \underline{If $b_{i-1} = \zeta_e^{k'}$ with $1 \leq k' \leq e-1$,}\\
we have $R\!E_{i}(w) = \boldsymbol{s}_{i} \cdots \boldsymbol{s}_3 \boldsymbol{t}_k \boldsymbol{t}_0 \boldsymbol{s}_{3} \cdots  \boldsymbol{s}_{c'}$ and\\
$R\!E_{i-1}(w) = \boldsymbol{s}_{i-1} \cdots \boldsymbol{s}_3 \boldsymbol{t}_{k'} \boldsymbol{t}_0 \boldsymbol{s}_{3} \cdots \boldsymbol{s}_{c}$. \\
Also, we have $R\!E_{i}(xw) = \boldsymbol{s}_{i} \cdots \boldsymbol{s}_3 \boldsymbol{t}_{k'} \boldsymbol{t}_0 \boldsymbol{s}_{3} \cdots \boldsymbol{s}_{c}$ and\\
$R\!E_{i-1}(xw) = \boldsymbol{s}_{i-1} \cdots \boldsymbol{s}_3 \boldsymbol{t}_k \boldsymbol{t}_0 \boldsymbol{s}_{3} \cdots \boldsymbol{s}_{c'-1}$.\\
It follows that $\boldsymbol{\ell}_1 = ((i-1)-3+1) +2+(c-3+1) + (i-3+1) +2+ (c'-3+1) = 2i+c+c'-5$ and $\boldsymbol{\ell}_2 = ((i-1)-3+1)+2+((c'-1)-3+1)+(i-3+1)+2+(c-3+1) = 2i+c+c'-6$ hence $\boldsymbol{\ell}_2 = \boldsymbol{\ell}_1 -1$.

\end{itemize}

It follows that

\begin{center}
if $a_i \neq 1$, then $\boldsymbol{\ell}(R\!E(s_iw)) = \boldsymbol{\ell}(R\!E(w)) -1. \hspace{1cm} (b)$
\end{center}

\end{itemize}

Suppose, on the other hand, $c_{i-1} > c_i$. Recall that $w' = s_iw$. If $w'[i-1,c'_{i-1}]$ and $w'[i,c'_{i}]$ denote the nonzero entries of $w'$ on the rows $i-1$ and $i$, respectively, we have $w'[i-1,c'_{i-1}] = a_i$ and $w'[i,c'_{i}] = a_{i-1}$. For $w'$, we have $c'_{i-1} < c'_{i}$, in which case the preceding analysis would give:

\begin{center}

if $a_{i-1} = 1$, then $\boldsymbol{\ell}(R\!E(s_i(s_iw))) = \boldsymbol{\ell}(R\!E(s_iw)) + 1$,\\
if $a_{i-1} \neq 1$, then $\boldsymbol{\ell}(R\!E(s_i(s_iw))) = \boldsymbol{\ell}(R\!E(s_iw)) - 1$.

\end{center}

Hence, since $s_i^2 = 1$, we get the following:

\begin{center}

if $a_{i-1} = 1$, then $\boldsymbol{\ell}(R\!E(s_iw)) = \boldsymbol{\ell}(R\!E(w)) - 1. \hspace{1cm} (a')$,\\
if $a_{i-1} \neq 1$, then $\boldsymbol{\ell}(R\!E(s_iw)) = \boldsymbol{\ell}(R\!E(w)) + 1. \hspace{1cm} (b')$.

\end{center}

\underline{Case 2: Suppose $x = t_i$ for $0 \leq i \leq e-1$.}\\

Set $w' := t_iw$. By the left multiplication by $t_i$, we have that the last $n-2$ rows of $w$ and $w'$ are the same. Hence, by Definition \ref{DefREiw} and Lemma \ref{LemmaBlocks}, we have:\\
$R\!E_3(xw)R\!E_4(xw) \cdots R\!E_n(xw) = R\!E_3(w)R\!E_4(w) \cdots R\!E_n(w)$. In order to prove our property in this case, we should compare $\boldsymbol{\ell}_1 := \boldsymbol{\ell}(R\!E_2(w))$ and $\boldsymbol{\ell}_2 := \boldsymbol{\ell}(R\!E_2(xw))$.

\begin{itemize}

\item \underline{Consider the case where $c_1 < c_2$.}\\

Since $c_1 < c_2$, by Lemma \ref{LemmaBlocks}, the blocks $w_2$ and $w'_2$ are of the form:

$w_2 = \begin{pmatrix}

b_1 & 0\\
0 & a_2\\

\end{pmatrix}$ and $w'_2 = \begin{pmatrix}

0 & \zeta_e^{-i}a_2\\
\zeta_e^{i}b_{1} & 0\\ 

\end{pmatrix}$ with $b_{1}$ instead of $a_{1}$ since $a_{1}$ may change when applying the algorithm if $c_{1} =1$.

\begin{itemize}
\item \underline{Suppose $a_2 = 1$,}\\
we have $b_1 = 1$ necessarily hence $\boldsymbol{\ell}_1 = 0$. Since $R\!E_2(xw) = \boldsymbol{t}_i$, we have $\boldsymbol{\ell}_2 = 1$. It follows that when $c_1 < c_2$,

\begin{center}
if $a_2 =1$, then $\boldsymbol{\ell}(R\!E(t_iw)) = \boldsymbol{\ell}(R\!E(w)) +1. \hspace{1cm} (c)$
\end{center}

\item \underline{Suppose $a_2 = \zeta_e^{k}$ with $1 \leq k \leq e-1$,} then $b_1 = \zeta_e^{-k}$.\\
We get $R\!E_2(w) = \boldsymbol{t}_k\boldsymbol{t}_0$. Thus, $\boldsymbol{\ell}_1 = 2$. We also get $R\!E_2(xw) = \mathbf{t}_{i-k}$. Thus, $\boldsymbol{\ell}_2 = 1$. It follows that when $c_1 < c_2$,

\begin{center}
if $a_2 \neq 1$, then $\boldsymbol{\ell}(R\!E(t_iw)) = \boldsymbol{\ell}(R\!E(w))  -1. \hspace{1cm} (d)$
\end{center}

\end{itemize}

\item \underline{Now, consider the case where $c_1 > c_2$.}\\

Since $c_1 > c_2$, by Lemma \ref{LemmaBlocks}, the blocks $w_2$ and $w'_2$ are of the form:

$w_2 = \begin{pmatrix}

0 & a_1\\
b_2 & 0\\

\end{pmatrix}$ and $w'_2 = \begin{pmatrix}

\zeta_e^{-i}b_2 & 0\\
0 & \zeta_e^{i}a_{1}\\ 

\end{pmatrix}$ with $b_{2}$ instead of $a_{2}$ since $a_{2}$ may change when applying the algorithm if $c_{2} =1$.

\begin{itemize}

\item \underline{Suppose $a_1 \neq \zeta_e^{-i}$,} then $b_2 \neq \zeta_e^i$.\\
We have $\boldsymbol{\ell}_1 =1$ necessarily, and since $\zeta_e^{i}a_1 \neq 1$, we have $\boldsymbol{\ell}_2 = 2$. Hence when $c_1 > c_2$,

\begin{center}
if $a_1 \neq \zeta_e^{-i}$, then $\boldsymbol{\ell}(R\!E(t_iw)) = \boldsymbol{\ell}(R\!E(w)) +1. \hspace{1cm} (e)$
\end{center}

\item \underline{Suppose $a_1 = \zeta_e^{-i}$,}\\
we have $\boldsymbol{\ell}_1 = 1$ and $\boldsymbol{\ell}_2 = 0$. Hence when $c_1 > c_2$,

\begin{center}
if $a_1 = \zeta_e^{-i}$, then $\boldsymbol{\ell}(R\!E(t_iw)) = \boldsymbol{\ell}(R\!E(w)) -1. \hspace{1cm} (f)$
\end{center}

\end{itemize}

\end{itemize}

This finishes our proof.
\end{proof}

\begin{proposition}\label{PropREwRedExp}

Let $w$ be an element of $G(e,e,n)$. The word $R\!E(w)$ is a reduced expression over $\mathbf{X}$ of $w$.

\end{proposition}

\begin{proof}

We must prove that $\ell(w) = \boldsymbol{\ell}(R\!E(w))$.\\
Let $\mathbf{x}_1\mathbf{x}_2 \cdots \mathbf{x}_r$ be a reduced expression over $\mathbf{X}$ of $w$. Hence $\ell(w) = \boldsymbol{\ell}(\mathbf{x}_1\mathbf{x}_2 \cdots \mathbf{x}_r) = r$. Since $R\!E(w)$ is a word representative over $\mathbf{X}$ of $w$, we have $\boldsymbol{\ell}(R\!E(w)) \geq \boldsymbol{\ell}(\mathbf{x}_1\mathbf{x}_2 \cdots \mathbf{x}_r) = r$. We prove that $\boldsymbol{\ell}(R\!E(w)) \leq r$. Observe that we can write $w$ as $x_1x_2 \cdots x_r$ where $x_1,x_2, \cdots, x_r$ are the matrices of $G(e,e,n)$ corresponding to $\mathbf{x}_1,\mathbf{x}_2, \cdots, \mathbf{x}_r$.\\
By Proposition \ref{prop.lengthcomp}, we have: $\boldsymbol{\ell}(R\!E(w)) = \boldsymbol{\ell}(R\!E(x_1x_2 \cdots x_r)) \leq \boldsymbol{\ell}(R\!E(x_2x_3 \cdots x_r)) +1 \leq \boldsymbol{\ell}(R\!E(x_3 \cdots x_r)) +2 \leq \cdots \leq r$. Hence $\boldsymbol{\ell}(R\!E(w)) = r = \ell(w)$ and we are done.

\end{proof}

The following proposition is useful in the next sections. Its proof is based on the proof of Proposition \ref{prop.lengthcomp}.

\begin{proposition}\label{PropLengthdecreas}

Let $w$ be an element of $G(e,e,n)$. Denote by $a_i$ the unique nonzero entry $w[i,c_i]$ on the row $i$ of $w$ where $1 \leq i, c_i \leq n$.

\begin{enumerate}

\item For $3 \leq i \leq n$, we have:
	\begin{enumerate}
		\item if $c_{i-1} < c_i$, then
				\begin{center}$\ell(s_{i}w) = \ell(w)-1$ if and only if $a_{i} \neq 1.$\end{center}
		\item if $c_{i-1} > c_i$, then
				\begin{center}$\ell(s_{i}w) = \ell(w)-1$ if and only if $a_{i-1} =1.$\end{center}
	\end{enumerate}
\item If $c_1 < c_2$, then $\forall\ 0 \leq k \leq e-1$, we have
				\begin{center}$\ell(t_k w) = \ell(w)-1$ if and only if $a_2 \neq 1.$\end{center}
\item If $c_1 > c_2$, then $\forall\ 0 \leq k \leq e-1$, we have
				\begin{center}$\ell(t_k w) = \ell(w)-1$ if and only if $a_1 = \zeta_{e}^{-k}.$\end{center}
				
\end{enumerate}

\end{proposition}

\begin{proof}

The claim $1(a)$ is deduced from $(a)$ and $(b)$, $1(b)$ is deduced from $(a')$ and $(b')$, $2$ is deduced from $(c)$ and $(d)$, and $3$ is deduced from $(e)$ and $(f)$ where $(a)$, $(b)$, $(a')$, $(b')$, $(c)$, $(d)$, $(e)$, and $(f)$ are given in the proof of Proposition \ref{prop.lengthcomp}.

\end{proof}

\begin{remark}

Proposition \ref{PropLengthdecreas} was useful in order to implement the interval Garside monoids that we will construct in Section \ref{SectionGarsideStructures}, see \cite{CHEVIEJMichel} and \cite{GAP3CPMichelNeaime}.

\end{remark}

\subsection{Elements of maximal length}

By using the geodesic normal forms of the elements of $G(e,e,n)$ defined by the algorithm, we characterize the elements that are of maximal length over the generating set of the presentation of Corran-Picantin of $G(e,e,n)$. We get the following.

\begin{proposition}\label{PropMaxLength}

Let $e > 1$ and $n \geq 2$. The maximal length of an element of $G(e,e,n)$ is $n(n-1)$. It is realized for diagonal matrices $w$ such that $w[i,i]$ is an $e$-th root of unity different from $1$ for $2 \leq i \leq n$. A minimal word representative of such an element is of the form $$(\mathbf{t}_{k_2}\mathbf{t}_0)(\mathbf{s}_3\mathbf{t}_{k_3}\mathbf{t}_0\mathbf{s}_3)\cdots (\mathbf{s}_n \cdots \mathbf{s}_3\mathbf{t}_{k_n}\mathbf{t}_0\mathbf{s}_3 \cdots \mathbf{s}_n),$$ with $1 \leq k_2, \cdots, k_n \leq e-1$. The number of elements of this form is $(e-1)^{(n-1)}$.

\end{proposition}

\begin{proof}

By the algorithm, an element $w$ in $G(e,e,n)$ is of maximal length when $w_i[i,i] = \zeta_e^{k}$ for $2 \leq i \leq n$ and $\zeta_e^{k} \neq 1$. By Lemma \ref{LemmaBlocks}, this condition is satisfied when $w$ is a diagonal matrix such that $w[i,i]$ is an $e$-th root of unity different from $1$ for $2 \leq i \leq n$. A minimal word representative given by the algorithm for such an element is of the form $(\mathbf{t}_{k_2}\mathbf{t}_0)(\mathbf{s}_3\mathbf{t}_{k_3}\mathbf{t}_0\mathbf{s}_3)\cdots (\mathbf{s}_n \cdots \mathbf{s}_3\mathbf{t}_{k_n}\mathbf{t}_0\mathbf{s}_3 \cdots \mathbf{s}_n)$ with $1 \leq k_2, \cdots, k_n \leq (e-1)$ which is of length $n(n-1)$. The number of elements of this form is $(e-1)^{(n-1)}$.

\end{proof}

\noindent Denote by $\lambda$ the element $\begin{pmatrix}

(\zeta_e^{-1})^{(n-1)} & & & \\
 & \zeta_e & & \\
 & & \ddots & \\
 & & & \zeta_e\\

\end{pmatrix} \in G(e,e,n)$.

\begin{example}\label{LemmaLengthlambda}

We have $R\!E(\lambda) = (\mathbf{t}_{1}\mathbf{t}_0)(\mathbf{s}_3\mathbf{t}_{1}\mathbf{t}_0\mathbf{s}_3)\cdots (\mathbf{s}_n \cdots \mathbf{s}_3\mathbf{t}_{1}\mathbf{t}_0\mathbf{s}_3 \cdots \mathbf{s}_n)$. Hence ${\ell}(\lambda)$ is equal to $n(n-1)$ which is the maximal length of an element of $G(e,e,n)$.

\end{example}

\begin{remark}

Consider the group $G(1,1,n)$, that is the symmetric group $S_n$. There exists a unique element of maximal length in $S_n$ that is of the form $$\mathbf{t}_0 (\mathbf{s}_3\mathbf{t}_0) \cdots (\mathbf{s}_{n-1}\cdots \mathbf{s}_3\mathbf{t}_0) (\mathbf{s}_{n}\cdots \mathbf{s}_3\mathbf{t}_0).$$ This corresponds to the maximal number of steps of the algorithm. The length of such an element is $n(n-1)/2$ which is already known for the symmetric group $S_n$, see Example 1.5.4 of \cite{GeckPfeifferBook}.

\end{remark}

\begin{remark}\label{RemUniqueMaximalG(2,2,n)}

Consider the group $G(2,2,n)$, that is the Coxeter group of type $D_n$. We have $e = 2$. Then, by Proposition \ref{PropMaxLength}, the number of elements of maximal length is equal to $(e-1)^{(n-1)} = 1$. Hence there exists a unique element of maximal length in $G(2,2,n)$. It is of the form $$(\mathbf{t}_{1}\mathbf{t}_0)(\mathbf{s}_3\mathbf{t}_{1}\mathbf{t}_0\mathbf{s}_3) \cdots (\mathbf{s}_n \cdots \mathbf{s}_3\mathbf{t}_{1}\mathbf{t}_0\mathbf{s}_3 \cdots \mathbf{s}_n).$$ The length of this element is $n(n-1)$ which is already known for Coxeter groups of type $D_n$, see Example 1.5.5 of \cite{GeckPfeifferBook}.

\end{remark}

\section{Balanced elements of maximal length}\label{SectionBalancedElements}

We prove that the only balanced elements of $G(e,e,n)$ that are of maximal length over $\mathbf{X}$ are $\lambda^k$ with $1 \leq k \leq e-1$, where $\lambda$ is the diagonal matrix such that $\lambda[i,i] = \zeta_e$ for $2 \leq i \leq n$ and $\lambda[1,1] = {(\zeta_e^{-1})}^{n-1}$, see Proposition \ref{PropMaxLength}. This is done by explicitly characterizing the intervals of the elements that are of maximal length. We start by defining two partial order relations on $G(e,e,n)$ as follows.

\begin{definition}\label{DefPartalOrder1}

Let $w, w' \in G(e,e,n)$. We say that $w'$ is a divisor of $w$ or $w$ is a multiple of $w'$, and write $w' \preceq w$, if $w = w' w''$ with $w'' \in G(e,e,n)$ and $\ell(w) = \ell(w') + \ell(w'')$. This defines a partial order relation on $G(e,e,n)$.

\end{definition}

\noindent Similarly, we consider another partial order relation on $G(e,e,n)$.

\begin{definition}\label{DefPartialOrder2}

Let $w$, $w' \in G(e,e,n)$. We say that $w'$ is a right divisor of $w$ or $w$ is a left multiple of $w'$, and write $w' \preceq_r w$, if there exists $w'' \in G(e,e,n)$ such that $w=w''w'$ and $\ell(w) = \ell(w'')+\ell(w')$.

\end{definition}

\begin{lemma}\label{LemmaCaractDivlambda}

Let $w, w' \in G(e,e,n)$ and let $\mathbf{x}_1\mathbf{x}_2 \cdots \mathbf{x}_r$ be a reduced expression over $\mathbf{X}$ of $w'$. We have $w' \preceq w$ if and only if,  for all $i$ s.t. $1 \leq i \leq r$, $\ell(x_ix_{i-1}\cdots x_1 w)=\ell(x_{i-1}\cdots x_1 w)-1.$ 

\end{lemma}

\begin{proof}

On the one hand, we have $w'w''=w$ with $w''=x_rx_{r-1} \cdots x_1 w$ and the condition $\forall\ 1 \leq i \leq r, \ell(x_ix_{i-1}\cdots x_1w)=\ell(x_{i-1}\cdots x_1 w)-1$ implies that \mbox{$\ell(w'')=\ell(w)-r$}. So we get $\ell(w'')+\ell(w')=\ell(w)$. Hence $w' \preceq w$.

On the other hand, since $x^2=1$ for all $x \in X$, we have $\ell(xw)=\ell(w) \pm 1$ for all \mbox{$w \in G(e,e,n)$}. If there exists $i$ such that $\ell(x_i x_{i-1}\cdots x_1 w) = \ell(x_{i-1} \cdots x_1 w)+1$ with $1 \leq i \leq r$, then $\ell(w'') = \ell(x_r x_{r-1}\cdots x_1 w) > \ell(w) -r$. It follows that $\ell(w') + \ell(w'') > \ell(w)$. Hence $w' \npreceq w$.

\end{proof}

Consider the homomorphism ${}^-$: $\mathbf{X}^{*} \longrightarrow G(e,e,n): \mathbf{x} \longmapsto \overline{\mathbf{x}} := x \in X$. If $R\!E(w) = \boldsymbol{x}_1\boldsymbol{x}_2 \cdots \boldsymbol{x}_r$ with $w \in G(e,e,n)$ and $\boldsymbol{x}_1,\boldsymbol{x}_2, \cdots, \boldsymbol{x}_r \in \mathbf{X}$, then $\overline{R\!E(w)} =  x_1x_2 \cdots x_r  = w$ where $x_1, x_2, \cdots, x_r \in X$.\\

In the sequel, we fix $1 \leq k \leq e-1$ and let $w \in G(e,e,n)$.

\begin{definition}

Let $\lambda$ be the diagonal matrix of $G(e,e,n)$ such that $\lambda[i,i] = \zeta_e$ for $2 \leq i \leq n$ and $\lambda[1,1] = (\zeta_e^{-1})^{n-1}$. We define $D_k$ to be the set \begin{center} $\left\{w \in G(e,e,n)\ s.t.\ \overline{R\!E_i(w)} \preceq \overline{R\!E_i(\lambda^k)}\ for\ 2 \leq i \leq n \right\}$, \end{center}
where $R\!E_i(w)$ is given in Definition \ref{DefREiw}.

\end{definition}

\begin{proposition}\label{PropCaractD}

The set $D_k$ consists of the elements $w$ of $G(e,e,n)$ such that for all \mbox{$2 \leq i \leq n$}, $R\!E_i(w)$ can be any of the following words:

\begin{tabular}{ll}
$\mathbf{s}_i \cdots \mathbf{s}_{i'}$ & with $2 \leq i' \leq i$,\\
$\mathbf{s}_i \cdots \mathbf{s}_3 \mathbf{t}_{k'}$ & with $0 \leq k' \leq e-1$,\\
$\mathbf{s}_i \cdots \mathbf{s}_3\mathbf{t}_{k}\mathbf{t}_0\mathbf{s}_3 \cdots \mathbf{s}_{i'}$ & with $2 \leq i' \leq i$.

\end{tabular}

\end{proposition}

\begin{proof}

We have $R\!E_i(\lambda^k) = \mathbf{s}_i \cdots \mathbf{s}_3 \mathbf{t}_k \mathbf{t}_0 \mathbf{s}_3 \cdots \mathbf{s}_i$. Let $w \in G(e,e,n)$. Note that $R\!E_i(w)$ is necessarily one of the words given in the first column of the following table. For each $R\!E_i(w)$, there exists a unique $w' \in G(e,e,n)$ with $R\!E(w')$ given in the second column, such that $\overline{R\!E_i(w)}w' = \overline{R\!E_i(\lambda^k)}$.

For $R\!E_i(w) = \mathbf{s}_i \cdots \mathbf{s}_{i'}$ with $2 \leq i' \leq i$, we get $R\!E(w') = \mathbf{s}_{i'-1} \cdots \mathbf{s}_3 \mathbf{t}_k \mathbf{t}_0 \mathbf{s}_3 \cdots \mathbf{s}_i$.

For $R\!E_i(w) = \mathbf{s}_i \cdots \mathbf{s}_3 \mathbf{t}_{k'}$ with $0 \leq k' \leq e-1$, we get $R\!E(w') = \mathbf{t}_{k'-k}\mathbf{s}_3 \cdots \mathbf{s}_i$. In this case, $\overline{R\!E_i(w)}\ \overline{R\!E(w')} = s_i \cdots s_3 t_{k'}t_{k'-k} s_3 \cdots s_i = s_i \cdots s_3 t_{k}t_{0} s_3 \cdots s_i = \overline{R\!E_i(\lambda^k)}$.

For $R\!E_i(w) = \mathbf{s}_i \cdots \mathbf{s}_3\mathbf{t}_k\mathbf{t}_0\mathbf{s}_3 \cdots \mathbf{s}_{i'}$ with $2 \leq i' \leq i$, we get $R\!E(w') = \mathbf{s}_{i'+1} \cdots \mathbf{s}_i$.

Finally, for $R\!E_i(w) = \mathbf{s}_i \cdots \mathbf{s}_3 \mathbf{t}_{k'} \mathbf{t}_0 \mathbf{s}_3 \cdots \mathbf{s}_{i'}$ with $1 \leq k' \leq e-1$, $k' \neq k$, and $2 \leq i' \leq i$, we get $R\!E(w') = \mathbf{s}_{i'} \cdots \mathbf{s}_3 \mathbf{t}_{k-k'} \mathbf{t}_{0} \mathbf{s}_3 \cdots \mathbf{s}_i$.

In the last column, we compute $\ell\left(\overline{R\!E_i(w)}\right) + \ell(\overline{R\!E(w')})$. It is equal to $\ell\left(\overline{R\!E_i(\lambda^k)}\right) = 2(i-1)$ only for the first three cases. The result follows immediately.\\

\begin{tabular}{|l|l|l|}
\hline
$R\!E_i(w)$ & $R\!E(w')$ & \\
\hline
$\mathbf{s}_i \cdots \mathbf{s}_{i'}$ with $2 \leq i' \leq i$ & $\mathbf{s}_{i'-1} \cdots \mathbf{s}_3 \mathbf{t}_k \mathbf{t}_0 \mathbf{s}_3 \cdots \mathbf{s}_i$ & $2(i-1)$\\
\hline
$\mathbf{s}_i \cdots \mathbf{s}_3 \mathbf{t}_{k'}$ with $0 \leq k' \leq e-1$ & $\mathbf{t}_{k'-k}\mathbf{s}_3 \cdots \mathbf{s}_i$ & $2(i-1)$\\
\hline
$\mathbf{s}_i \cdots \mathbf{s}_3\mathbf{t}_k \mathbf{t}_0\mathbf{s}_3 \cdots \mathbf{s}_{i'}$ with $2 \leq i' \leq i$ & $\mathbf{s}_{i'+1} \cdots \mathbf{s}_i$ & $2(i-1)$\\
\hline
$\mathbf{s}_i \cdots \mathbf{s}_3 \mathbf{t}_{k'} \mathbf{t}_0 \mathbf{s}_3 \cdots \mathbf{s}_{i'}$ with $1 \leq k' \leq e-1$, & $\mathbf{s}_{i'} \cdots \mathbf{s}_3 \mathbf{t}_{k-k'} \mathbf{t}_{0} \mathbf{s}_3 \cdots \mathbf{s}_i$ & $2(i-1)+$\\
$k' \neq k$, and $2 \leq i' \leq i$ & & $2(i'-1)$ \\
\hline

\end{tabular}

\end{proof}

The next proposition characterizes the divisors of $\lambda^k$ in $G(e,e,n)$.

\begin{proposition}\label{PropDisDivLambda}

The set $D_k$ is equal to the interval $[1,\lambda^k]$, where \begin{center}$[1,\lambda^k] = \left\{ w \in G(e,e,n)\ s.t.\ 1 \preceq w \preceq \lambda^k  \right\}$.\end{center}

\end{proposition}

\begin{proof}

Let $w \in G(e,e,n)$. We have $R\!E(w) = R\!E_2(w)R\!E_3(w) \cdots R\!E_n(w)$. Let $\textbf{w} \in \mathbf{X}^{*}$ be a word representative of $w$. Denote by $\overleftarrow{\mathbf{w}} \in \mathbf{X}^{*}$ the word obtained by reading $\textbf{w}$ from right to left. For $3 \leq i \leq n$, we denote by $\alpha_i$ the element that corresponds to \begin{small}$\overline{\overleftarrow{R\!E_{i-1}(w)}} \cdots \overline{\overleftarrow{R\!E_2(w)}}$\end{small} in $G(e,e,n)$.

Suppose that $w \in D_k$. We apply Lemma  \ref{LemmaCaractDivlambda} to prove that $w \preceq \lambda^k$. Fix $2 \leq i \leq n$. By Proposition \ref{PropCaractD}, we have three different possibilities for $R\!E_i(w)$.

First, consider the cases $R\!E_i(w) = \mathbf{s}_i \cdots  \mathbf{s}_3  \mathbf{t}_k  \mathbf{t}_0  \mathbf{s}_3 \cdots  \mathbf{s}_{i'}$ or $\mathbf{s}_i \cdots \mathbf{s}_{i'}$ with $2 \leq i' \leq i$. Hence $\overleftarrow{R\!E_i(w)} = \mathbf{s}_{i'} \cdots \mathbf{s}_3 \mathbf{t}_0 \mathbf{t}_k \mathbf{s}_3 \cdots \mathbf{s}_i$ or $\mathbf{s}_{i'} \cdots \mathbf{s}_i$, respectively.
Note that the left multiplication of the matrix $\lambda^k$ by $\alpha_i$ produces permutations only in the block consisting of the first $i-1$ rows and the first $i-1$ columns of $\lambda^k$. Since  $\lambda[i,i] = \zeta_e^k$ ($\neq 1$), by $1(a)$ of Proposition \ref{PropLengthdecreas}, the left multiplication of $\alpha_i \lambda^k$ by $s_{i'} \cdots s_i$ decreases the length maximally $-$ that is, each generator causes a decrease of length $1$. Now, by $2$ of Proposition \ref{PropLengthdecreas}, the left multiplication of $s_3 \cdots s_i \alpha_i \lambda^k$ by $t_k$ decreases the length of $1$. Note that by these left multiplications, $\lambda^k[i,i] = \zeta_e^k$ is shifted to the first row then transformed to $ \zeta_e^k  \zeta_e^{-k}=1$. Hence, by $1(b)$ of Proposition \ref{PropLengthdecreas}, the left multiplication of $t_1s_3 \cdots s_i \alpha_i \lambda^k$ by $s_{i'} \cdots s_3 t_0$ decreases the length maximally. Thus, by Lemma \ref{LemmaCaractDivlambda}, we have $w \preceq \lambda^k$.

Suppose that $R\!E_i(w) = \mathbf{s}_i \cdots \mathbf{s}_3 \mathbf{t}_{k'}$ with $0 \leq k' \leq e-1$. We have $\overleftarrow{R\!E_i(w)} = \mathbf{t}_{k'} \mathbf{s}_3 \cdots \mathbf{s}_i$.
Since $\lambda^k[i,i] = \zeta_e^k$ ($\neq 1$), by $1(a)$ of Proposition \ref{PropLengthdecreas}, the left multiplication of $\alpha_i \lambda^k$ by $s_3 \cdots s_i$ decreases the length maximally. By $2$ of Proposition \ref{PropLengthdecreas}, the left multiplication of $s_3 \cdots s_i \alpha_i \lambda^k$ by $t_{k'}$ also decreases the length of $1$. Hence, by applying Lemma \ref{LemmaCaractDivlambda}, we have $w \preceq \lambda^k$.

Conversely, suppose that $w \notin D_k$, we prove that $w \npreceq \lambda^k$.
If $R\!E(w) = \mathbf{x}_1 \cdots \mathbf{x}_r$, by Lemma \ref{LemmaCaractDivlambda}, we show that there exists $1 \leq i \leq r$ such that $\ell(x_i x_{i-1} \cdots x_1 \lambda^k) = \ell(x_{i-1} \cdots x_1 \lambda^k) +1$. Since $w \notin D$, by Proposition \ref{PropCaractD}, we may consider the first $R\!E_i(w)$ that appears in $R\!E(w) = R\!E_2(w) \cdots R\!E_n(w)$ such that $R\!E_i(w) =\\ \mathbf{s}_i \cdots \mathbf{s}_3 \mathbf{t}_{k'} \mathbf{t}_0 \mathbf{s}_3 \cdots \mathbf{s}_{i'}$ with $1 \leq k' \leq e-1$, $k' \neq k$, and $2 \leq i' \leq i$. Thus, we have $\overleftarrow{{R\!E}_i(w)} = \mathbf{s}_{i'} \cdots \mathbf{s}_3 \mathbf{t}_0\mathbf{t}_{k'} \mathbf{s}_3 \cdots \mathbf{s}_i$.
Since $\lambda^k[i,i] = \zeta_e^k$ ($\neq 1$), by $1(a)$ of Proposition \ref{PropLengthdecreas}, the left multiplication of $\alpha_i \lambda^k$ by $s_3 \cdots s_i$ decreases the length maximally. By $2$ of Proposition \ref{PropLengthdecreas}, the left multiplication of $s_3 \cdots s_i \alpha_i \lambda^k$ by $t_{k'}$ also decreases the length of $1$. Note that by these left multiplications, $\lambda^k[i,i] = \zeta_e^k$ is shifted to the first row then transformed to $\zeta_e^k \zeta_e^{-k'} = \zeta_e^{k-k'}$. Since $k \neq k'$, we have $\zeta_e^{k-k'} \neq 1$. By $3$ of Proposition \ref{PropLengthdecreas}, it follows that the left multiplication of $t_{k'} s_3 \cdots s_i \alpha_i \lambda^k$ by $t_0$ increases the length. Hence $w \npreceq \lambda^k$.

\end{proof}

We want to recognize if an element $w \in G(e,e,n)$ is in the set $D_k$ directly from its matrix form. For this purpose, we introduce nice combinatorial tools defined as follows. Fix $1 \leq k \leq e-1$. Let $w \in G(e,e,n)$.

\begin{definition}\label{DefinitionBullets}

An index $[i,c]$ is said to be a \emph{bullet} if $w[j,d] = 0$ for all $[j,d] \in \left\{ [j,d]\ s.t.\ j \leq i\ and\ d\leq c  \right\} \setminus \left\{[i,c] \right\} $. When $[i,c]$ is a bullet, $w[i,c]$ is represented by an encircled element.

\end{definition}

\begin{definition}\label{DefinitionZwZ'w}
We define two sets of matrix indices $Z(w)$ and $Z'(w)$ as follows.
\begin{itemize}

\item $Z(w) := \left\{ [j,d]\ s.t.\ j \leq i\ and\ d \leq c\ for\ some\ bullet\ [i,c] \right\}$.
\item $Z'(w)$ is the set of matrix indices not in $Z(w)$.

\end{itemize}

\end{definition}

We draw a path in the matrix $w$ that separates it into two parts such that the upper left-hand side is $Z(w)$ and the other side is $Z'(w)$. Let us illustrate this by the following example.

\begin{example}


Let $w =
\left(
\begin{BMAT}{ccccc}{ccccc}
0 & 0 & 0  & \encircle{$\zeta_3^2$} & 0 \\
0 & 0 & 0 & 0 & \zeta_3 \\
0 & 0 & \encircle{$\zeta_3$} & 0 & 0 \\
\encircle{$1$} & 0 & 0 & 0 & 0 \\
0 & \zeta_3^2 & 0 & 0 & 0
\addpath{(0,1,0)rurruurur}
\end{BMAT}
\right) \in G(3,3,5).
$ When $[i,c]$ is a bullet, $w[i,c]$ is an encircled element and the drawn path separates $Z(w)$ from $Z'(w)$.

\end{example}

\begin{remark}\label{RemBulletsinZ}

Let $[i,c]$ be one of the bullets of $w \in G(e,e,n)$. We have
\begin{center} $[i-1,c] \in Z(w)$ and $[i,c-1] \in Z(w)$. \end{center}
An index $[i,c]$ such that $w[i,c] \neq 0$ and $[i,c]$ is not a bullet does not satisfy this condition.

\end{remark}

\begin{remark}\label{RemBulletsFirstRowColumn}

The indices corresponding to nonzero entries on the first row and the first column of $w$ are always bullets. In particular, when $w[1,1] \neq 0$, we have $[1,1]$ is a bullet and it is the only bullet of $w$ (as this nonzero entry at $[1,1]$ is above, or to the left of, every entry of $w$).  

\end{remark}

The following proposition gives a nice description of the divisors of $\lambda^k$ in $G(e,e,n)$.

\begin{proposition}\label{PropZ'=1orZetae}

Let $w \in G(e,e,n)$. We have that $w \in D_k$ (i.e. $w \preceq \lambda^k$) if and only if, for all $[j,d] \in Z'(w)$, $w[j,d]$ is either $0$, $1$, or $\zeta_e^k$.

\end{proposition}

\begin{proof}

Let $w \in D_k$ and let $w[i,c] \neq 0$ for $[i,c] \in Z'(w)$. Since $w \in D_k$, by Proposition \ref{PropCaractD}, we have $R\!E_i(w) = \mathbf{s}_i \cdots \mathbf{s}_{i'}$ or $\mathbf{s}_i \cdots \mathbf{s}_3\mathbf{t}_k\mathbf{t}_0\mathbf{s}_3 \cdots \mathbf{s}_{i'}$ for $2 \leq i' \leq i$ (the case $R\!E_i(w) = \mathbf{s}_i \cdots \mathbf{s}_3\mathbf{t}_{k'}$ for $0 \leq k' \leq e-1$ appears only when $w[i,c] \neq 0$ and $[i,c]$ is a bullet). By Lemma \ref{LemmaBlocks}, we have $w[i,c] = w_i[i,d]$ for some $1 < d \leq i$. It follows that for $R\!E_i(w) = \mathbf{s}_i \cdots \mathbf{s}_{i'}$, we have $w[i,c] = 1$ and for $R\!E_i(w) = \mathbf{s}_i \cdots \mathbf{s}_3\mathbf{t}_k\mathbf{t}_0\mathbf{s}_3 \cdots \mathbf{s}_{i'}$, we have $w[i,c] = \zeta_e^k$.

Conversely, suppose that $w[j,d]$ is $0$, $1$, or $\zeta_e^k$ whenever $[j,d] \in Z'(w)$. Firstly, consider a nonzero entry $w[i,c]$ of $w$ for which $[i,c] \in Z'(w)$. From Remark \ref{RemBulletsFirstRowColumn}, we have $i \geq 2$. Once again $R\!E_i(w) = \mathbf{s}_i \cdots \mathbf{s}_3\mathbf{t}_k\mathbf{t}_0\mathbf{s}_3 \cdots \mathbf{s}_{i'}$ or $\mathbf{s}_i \cdots \mathbf{s}_{i'}$ for $2 \leq i' \leq i$. On the other hand, if $w[i,c]$ is a nonzero entry of $w$ for which $[i,c] \notin Z'(w)\ -$ that is, $[i,c]$ is a bullet of $w$, so by Lemma \ref{LemmaBlocks}, we have $w_i[i,1] = \zeta_e^{k'}$ for some $0 \leq k' \leq e-1$, for which case $R\!E_{i}(w) = \mathbf{s}_i \cdots \mathbf{s}_3 \mathbf{t}_{k'}$. Hence, by Proposition \ref{PropCaractD}, we have $w \in D_k$.

\end{proof}

\begin{example}

Let $w =
\left(
\begin{BMAT}{ccccc}{ccccc}
0 & 0 & 0  & \encircle{$1$} & 0 \\
0 & 0 & 0 & 0 & \boxed{\zeta_3} \\
0 & 0 & \encircle{$\zeta_3$} & 0 & 0 \\
\encircle{$\zeta_3$} & 0 & 0 & 0 & 0 \\
0 & \boxed{1} & 0 & 0 & 0
\addpath{(0,1,0)rurruurur}
\end{BMAT}
\right) \in G(3,3,5)$.\\ For all $[i,c] \in Z'(w)$, we have $w[i,c]$ is equal to $1$ or $\zeta_3$ (these are the boxed entries of $w$). It follows immediately that $w \preceq \lambda$ ($w \in [1,\lambda]$).

\end{example}

\begin{example}

Let $w =
\left(
\begin{BMAT}{ccccc}{ccccc}
0 & 0 & 0  & \encircle{$1$} & 0 \\
0 & 0 & 0 & 0 & \boxed{1} \\
0 & 0 & \encircle{$\zeta_3$} & 0 & 0 \\
\encircle{$1$} & 0 & 0 & 0 & 0 \\
0 & \boxed{\zeta_3^2} & 0 & 0 & 0
\addpath{(0,1,0)rurruurur}
\end{BMAT}
\right) \in G(3,3,5).
$\\ For all $[i,c] \in Z'(w)$, we have $w[i,c]$ is equal to $1$ or $\zeta_3^2$ (these are the boxed entries of $w$). It follows immediately that $w \in [1,\lambda^2]$.

\end{example}

\begin{example}

Let $w =
\left(
\begin{BMAT}{cccc}{cccc}
\encircle{$\zeta_3^2$} & 0 & 0  & 0 \\
0 & 0 & \zeta_3 & 0 \\
0 & \zeta_3 & 0 & 0 \\
0 & 0 & 0 & \boxed{\zeta_3^2} 
\addpath{(0,3,0)rurrr}
\end{BMAT}
\right) \in G(3,3,4).$\\ There exists $[i,c] \in Z'(w)$ such that $w[i,c] = \zeta_3^2$ (the boxed element in $w$). It follows immediately that $w \notin [1,\lambda]$. Moreover, there exists $[i',c'] \in Z'(w)$ such that $w[i',c'] = \zeta_3$. Hence we have $w \notin [1,\lambda^2]$. 

\end{example}

\begin{remark}\label{RemDivisorsLambdaG(2,2,n)}

Let $w$ be an element of the Coxeter group $G(2,2,n)$. The nonzero elements $w[i,c]$ with $[i,c] \in Z'(w)$ are always equal to $1$ or $-1$. Hence by Proposition \ref{PropZ'=1orZetae}, all the elements of $G(2,2,n)$ are left divisors of the unique element of maximal length $\lambda$ in $G(2,2,n)$, see Remark \ref{RemUniqueMaximalG(2,2,n)}.\\ For example,
Let $w =
\left(
\begin{BMAT}{ccccc}{ccccc}
0 & 0 & 0  & \encircle{$1$} & 0 \\
0 & 0 & 0 & 0 & \boxed{1} \\
0 & 0 & \encircle{$-1$} & 0 & 0 \\
\encircle{$1$} & 0 & 0 & 0 & 0 \\
0 & \boxed{-1} & 0 & 0 & 0
\addpath{(0,1,0)rurruurur}
\end{BMAT}
\right)$ be an element of $G(2,2,4)$. Since all the nonzero elements $w[i,c]$ with $[i,c] \in Z'(w)$ are equal to $1$ or $-1$, it follows immediately that $w \preceq \lambda$.

\end{remark}

\medskip

Our description of the interval $[1, \lambda^k]$ allows us to prove easily that $\lambda^k$ is balanced. Let us recall the definition of a balanced element.

\begin{definition}\label{DefBalancedElement}

A balanced element in $G(e,e,n)$ is an element $w$ such that $w' \preceq w$ holds precisely when $w \preceq_r w'$. 

\end{definition}

\noindent The next lemma is obvious.

\begin{lemma}\label{LemmaDivisorsForInduction}

Let $g$ be a balanced element and let $w, w' \in [1,g]$. If $w' \preceq w$, then $(w')^{-1}w \preceq g$.

\end{lemma}

In order to prove that $\lambda^k$ is balanced, we first check the following.

\begin{lemma}\label{LemmaforPropBalanced}

If $w \in D_k$, we have $w^{-1}\lambda^k \in D_k$ and $\lambda^k w^{-1} \in D_k$.

\end{lemma}

\begin{proof}

Let $w \in D_k$. We show that $w^{-1}\lambda^k = \overline{t_w}\lambda^k \in D_k$ and $\lambda^k w^{-1} = \lambda^k \overline{t_w} \in D_k$, where $\overline{t_w}$ is the complex conjugate of the transpose $t_w$ of the matrix $w$. We use the matrix form of an element of $D_k$. If $[i,c]$ is a bullet of $w$, then $[c,i]$ is a bullet of $\overline{t_w} = w^{-1}$ and $w^{-1}[c,i] = \overline{w[i,c]}$. Then, if $[i,c] \in Z'(w^{-1})$, we have $[c,i] \in Z'(w)$. Since $w \in D_k$, we have $w[c,i] \in \left\{ 0,1,\zeta_e^k \right\}$ whenever $[c,i] \in Z'(w)$. Then $w^{-1}[i,c] \in \left\{ 0,1,\zeta_e^{-k} \right\}$ whenever $[i,c] \in Z'(w^{-1})$. Multiplying $w^{-1}$ by $\lambda^k$, we get that $(w^{-1}\lambda^k)[i,c]$ and $(\lambda^k w^{-1})[i,c]$ are equal to $0$, $1$, or $\zeta_e^k$ whenever $[i,c] \in Z'(w^{-1}\lambda^k)$ and $Z'(\lambda^kw^{-1})$. Hence $w^{-1}\lambda^k$ and $\lambda^kw^{-1}$ belong to $D_k$.

\end{proof}

\begin{example}

We illustrate the idea of the proof of Lemma \ref{LemmaforPropBalanced} for $k=1$.\\
Consider $w \in D_1$ as follows and show that $\overline{t_w}\lambda \in D_1$:
$w =
\left(
\begin{BMAT}{ccccc}{ccccc}
0 & 0 & 0  & 0 & \encircle{$1$} \\
0 & \encircle{$1$} & 0 & 0 & 0 \\
0 & 0 & 0 & \encircle{$\zeta_3$} & 0 \\
\encircle{$\zeta_3^2$} & 0 & 0 & 0 & 0 \\
0 & 0 & \boxed{1} & 0 & 0
\addpath{(0,1,0)ruururrru}
\end{BMAT}
\right)$ $\overset{t_w}{\longrightarrow}$ $\left(
\begin{BMAT}{ccccc}{ccccc}
0 & 0 & 0  & \encircle{$\zeta_3^2$} & 0 \\
0 & \encircle{$1$} & 0 & 0 & 0 \\
0 & 0 & 0 & 0 & \boxed{1} \\
0 & 0 &  \boxed{\zeta_3} & 0 & 0 \\
\encircle{$1$} & 0 & 0 & 0 & 0
\addpath{(0,0,0)ruuururrur}
\end{BMAT}
\right)$ $\overset{\overline{t_w}}{\longrightarrow}$ $\left(
\begin{BMAT}{ccccc}{ccccc}
0 & 0 & 0  & \encircle{$\zeta_3$} & 0 \\
0 & \encircle{$1$} & 0 & 0 & 0 \\
0 & 0 & 0 & 0 & \boxed{1} \\
0 & 0 & \boxed{\zeta_3^2} & 0 & 0 \\
\encircle{$1$} & 0 & 0 & 0 & 0
\addpath{(0,0,0)ruuururrur}
\end{BMAT}
\right)$ $\overset{\overline{t_w}\lambda}{\longrightarrow}$ $\left(
\begin{BMAT}{ccccc}{ccccc}
0 & 0 & 0  & \encircle{$1$} & 0 \\
0 & \encircle{$\zeta_3$} & 0 & 0 & 0 \\
0 & 0 & 0 & 0 & \boxed{\zeta_3} \\
0 & 0 & \boxed{1} & 0 & 0 \\
\encircle{$\zeta_3$} & 0 & 0 & 0 & 0
\addpath{(0,0,0)ruuururrur}
\end{BMAT}
\right)$.

\end{example}

\medskip

\begin{proposition}\label{Propbalanced}

The element $\lambda^k$ is balanced.

\end{proposition}

\begin{proof}

Suppose that $w \preceq \lambda^k$. By Proposition \ref{PropDisDivLambda}, we have $w \in D_k$, so $\lambda^k w^{-1}$ is in $D_k$ by Lemma \ref{LemmaforPropBalanced}. Hence $\lambda^k = (\lambda^k w^{-1})w$ satisfies $\ell(\lambda^k w^{-1}) + \ell(w) = \ell(\lambda^k)$, namely $w \preceq_r \lambda^k$. 

Conversely, suppose that $w \preceq_r \lambda^k$. We have $\lambda^k = w' w$ with $w' \in G(e,e,n)$ and $\ell(w') + \ell(w) = \ell(\lambda^k)$. It follows that $w' \in D_k$, then ${w'}^{-1} \lambda^k \in D_k$ by Lemma \ref{LemmaforPropBalanced}. Since  $w = {w'}^{-1}\lambda^k$, we have $w \in D_k$, namely $w \preceq \lambda^k$.

\end{proof}

In the following theorem, we show that the elements $\lambda^k$ are the only balanced elements of maximal length of $G(e,e,n)$ for $1 \leq k \leq e-1$.


\begin{theorem}\label{PropAllBalancedofMaxLength}

The balanced elements of $G(e,e,n)$  that are of maximal length are precisely $\lambda^k$ with $1 \leq k \leq e-1$. The set $D_k$ of the divisors of each $\lambda^k$ is characterized in Propositions \ref{PropCaractD} and \ref{PropZ'=1orZetae}.

\end{theorem}

\begin{proof}

Let $w \in G(e,e,n)$ be an element of $G(e,e,n)$ of maximal length, namely by Proposition \ref{PropMaxLength}, a diagonal matrix such that for $2 \leq i \leq n$, $w[i,i]$ is an $e$-th root of unity different from $1$. Analogously to Proposition \ref{PropZ'=1orZetae}, a left divisor $w'$ of $w$ satisfies that for all $2 \leq i \leq n$, if $[i,c]$ is not a bullet of $w'$, then $w'[i,c]$ is either $0$, $1$, or $w[i,i]$.

By Proposition \ref{Propbalanced}, we already have that $\lambda^k$ is balanced for $1 \leq k \leq e-1$. Suppose that $w$ is of maximal length such that $w[i,i] \neq w[j,j]$ for $2 \leq i,j \leq n$ and $i \neq j$. Let $s_{ij}$ be the transposition matrix. We have $s_{ij}[1,1]$ is nonzero and so, by Remark \ref{RemBulletsFirstRowColumn}, $[1,1]$ is the unique bullet of $s_{ij}$. Hence if $[j,d]$ is not a bullet of $s_{ij}$, then $s_{ij}[j,d]$ is either $0$ or $1$. So, $s_{ij}$ left-divides $w$. Hence $w' := s_{ij}^{-1}w = s_{ij}w$ right-divides $w$. We also have $w'[1,1]$ is nonzero, and so $[1,1]$ is the unique bullet of $w'$. Thus, $[j,i]$ is not a bullet and $w'[j,i] = w[i,i]$ is neither $0$, $1$, nor $w[j,j]$ (which was assumed different from $w[i,i]$). So $w' \npreceq w$, and so $w$ is not balanced.

It follows that the balanced elements of $G(e,e,n)$ that are of maximal length are precisely $\lambda^k$ with $1 \leq k \leq e-1$.

\end{proof}

\section{Interval Garside structures}\label{SectionGarsideStructures}

In this section, we construct the monoid $M([1,\lambda^k])$ associated with each of the intervals $[1,\lambda^k]$ constructed in the previous section with $1 \leq k \leq e-1$. By Proposition \ref{PropAllBalancedofMaxLength}, $\lambda^k$ is balanced. Hence, by Theorem \ref{TheoremMichelGarside}, in order to prove that $M([1,\lambda^k])$ is a Garside monoid, it remains to show that both posets $([1,\lambda^k],\preceq)$ and $([1,\lambda^k],\preceq_r)$ are lattices. This is stated in Corollary \ref{CorBothPosetsLattices}. The interval structures are given \mbox{in Theorem \ref{TheoremIntervalStructure}.}

\subsection{Least common multiples}

Let $1 \leq k \leq e-1$ and let $w \in [1,\lambda^k]$. For each $1 \leq i \leq n$ there exists a unique $c_i$ such that $w[i,c_i] \neq 0$. We denote $w[i,c_i]$ by $a_i$. We apply Lemma \ref{LemmaCaractDivlambda} to prove the following lemmas. The reader is invited to write the matrix form of $w$ to illustrate each step of the proof.

\begin{lemma}\label{tit0Divw}

Let $t_i \preceq w$ where $i \in \mathbb{Z}/e\mathbb{Z}$.

\begin{itemize}

\item If $c_1 < c_2$, then $t_kt_0 \preceq w$ and
\item if $c_2 < c_1$, then  $t_j \npreceq w$ for all $j$ with $j \neq i$.

\end{itemize}

Hence if $t_i \preceq w$ and $t_j \preceq w$ with $i,j \in \mathbb{Z}/e\mathbb{Z}$ and $i \neq j$, then $t_kt_0 \preceq w$.

\end{lemma}

\begin{proof}

\underline{Suppose $c_1 < c_2$.}\\
Since $a_1 = w[1,c_1]$ is nonzero , and above and to the left of $a_2 = w[2,c_2]$ (as $c_1 < c_2$), then $[2,c_2]$ is not a bullet. It belongs to $Z'(w)$. Since $w \in [1,\lambda^k]$ and $[2,c_2] \in Z'(w)$, by Proposition \ref{PropZ'=1orZetae}, $a_2 = 1$ or $\zeta_e^k$. Since $t_i \preceq w$, we have $\ell(t_iw) = \ell(w) - 1$. Hence by $2$ of \mbox{Proposition \ref{PropLengthdecreas}}, we get $a_2 \neq 1$. Hence $a_2 = \zeta_e^k$.
Again by $2$ of Proposition \ref{PropLengthdecreas}, since $a_2 \neq 1$, we have $\ell(t_kw) = \ell(w) - 1$. Let $w' := t_kw$. We have $w'[1,c_2] = \zeta_e^{-k}a_2 = \zeta_e^{-k} \zeta_e^{k} = 1$. Hence by $3$ of Proposition $\ref{PropLengthdecreas}$, $\ell(t_0w') = \ell(w') - 1$. It follows that $t_kt_0 \preceq w$.\\
\underline{Suppose $c_2 < c_1$.}\\
Since $t_i \preceq w$, we have $\ell(t_iw) = \ell(w) - 1$. Hence by $3$ of Proposition \ref{PropLengthdecreas}, we have $a_1 = \zeta_e^{-i}$. If there exists $j \in \mathbb{Z}/e\mathbb{Z}$ with $j \neq i$ such that $t_j \preceq w$, then $\ell(t_jw) = \ell(w) - 1$. Again by $3$ of Proposition \ref{PropLengthdecreas}, we have $a_1 = \zeta_e^{-j}$. Thus, $i = j$ which contradicts the hypothesis.\\
The last statement of the lemma follows immediately.

\end{proof}

\begin{lemma}\label{tis3Divw}

If $t_i \preceq w$ with $i \in \mathbb{Z}/e\mathbb{Z}$ and $s_3 \preceq w$, then $s_3t_is_3 = t_is_3t_i \preceq w$.

\end{lemma}

\begin{proof}

Set $w' := s_3w$ and $w'':=t_iw'$.\\
\underline{Suppose $c_1 < c_2$.}\\
Since $w \in [1,\lambda^k]$ and $[2,c_2] \in Z'(w)$, by Proposition \ref{PropZ'=1orZetae}, we get $a_2 = 1$ or $\zeta_e^k$. Since $t_i \preceq w$, we have $\ell(t_iw) = \ell(w) - 1$. Thus, by $2$ of Proposition \ref{PropLengthdecreas}, we get $a_2 \neq 1$. Hence $a_2 = \zeta_e^k$.\\
Suppose that $c_3 < c_2$. Since $s_3 \preceq w$, we have $\ell(s_3w) = \ell(w) - 1$. Hence by $1(b)$ of Proposition \ref{PropLengthdecreas}, $a_2 = 1$ which is not the case. So instead it must be that $c_2 < c_3$. Assume $c_2 < c_3$. Since $c_1 < c_2 < c_3$, $a_1 = w[1,c_1]$ is a nonzero entry which is above and to the left of $a_3 = w[3,c_3]$. Then $[3,c_3]$ is in $Z'(w)$. Since $w \in [1, \lambda^k]$ and $[3,c_3] \in Z'(w)$, we have $a_3 = 1$ or $\zeta_e^k$. By $1(a)$ of Proposition \ref{PropLengthdecreas}, we have $a_3 \neq 1$. Hence $a_3$ is equal to $\zeta_e^k$. Now, we prove that $s_3t_is_3 \preceq w$ by applying Lemma \ref{LemmaCaractDivlambda}. Indeed, since $a_3 \neq 1$, by $2$ of Proposition \ref{PropLengthdecreas}, we have $\ell(t_i w') = \ell(w') -1$, and since $a_2 \neq 1$, by $1(a)$ of Proposition \ref{PropLengthdecreas}, we have $\ell(s_3 w'') = \ell(w'') -1$.\\
\underline{Suppose $c_2 < c_1$.}\\
Since $\ell(t_iw) = \ell(w) -1$, by $3$ of Proposition \ref{PropLengthdecreas}, we have $a_1 = \zeta_e^{-i}$.
\begin{itemize}

\item \underline{Assume $c_2 < c_3$.}\\
Since $\ell(s_3 w) = \ell(w) -1$, by $1(a)$ of Proposition \ref{PropLengthdecreas}, we have $a_3 \neq 1$. We have $\ell(t_iw') = \ell(w') - 1$ for both cases $c_1 < c_3$ and $c_3 < c_1$. Actually, if $c_1 < c_3$, since $a_3 \neq 1$, by $2$ of Proposition \ref{PropLengthdecreas}, we have $\ell(t_iw') = \ell(w') -1$, and if $c_3 < c_1$, since $a_1 = \zeta_e^{-i}$, by $3$ of Proposition \ref{PropLengthdecreas}, $\ell(t_i w') = \ell(w') -1$. Now, since $\zeta_e^{i}a_1 = \zeta_e^{i}\zeta_e^{-i} = 1$, by $1(b)$ of Proposition \ref{PropLengthdecreas}, we have $\ell(s_3 w'') = \ell(w'') -1$.

\item \underline{Assume $c_3 < c_2$.}\\
Since $a_1 = \zeta_e^{-i}$, by $3$ of Proposition \ref{PropLengthdecreas}, $\ell(t_iw') = \ell(w') -1$. Since $\zeta_e^{i}a_1 = 1$, by $1(b)$ of Proposition \ref{PropLengthdecreas}, we have $\ell(s_3 w'') = \ell(w'') -1$.

\end{itemize}

\end{proof}

\begin{lemma}\label{tisjDivw}

If $t_i \preceq w$ with $i \in \mathbb{Z}/e\mathbb{Z}$ and $s_j \preceq w$ with $4 \leq j \leq n$, then \mbox{$t_i s_j = s_j t_i \preceq w$}.

\end{lemma}

\begin{proof}

We distinguish four different cases: case 1: $c_1 < c_2$ and $c_{j-1} < c_j$, case 2: $c_1 < c_2$ and $c_{j} < c_{j-1}$, case 3: $c_2 < c_1$ and $c_{j-1} < c_j$, and case 4: $c_2 < c_1$ and $c_{j} < c_{j-1}$. The proof is similar to the proofs of Lemmas \ref{tit0Divw} and \ref{tis3Divw} so we prove that $s_jt_i \preceq w$ only for the first case. Suppose that $c_1 < c_2$ and $c_{j-1} < c_j$. Since $t_i \preceq w$, we have $\ell(t_iw) = \ell(w) -1$. Hence, by $2$ of Proposition \ref{PropLengthdecreas}, we have $a_2 \neq 1$. Also, since $s_j \preceq w$, we have $\ell(s_jw) = \ell(w) -1$. Hence, by $1(a)$ of Proposition \ref{PropLengthdecreas}, we have $a_j \neq 1$. Set $w' := s_jw$. Since $a_2 \neq 1$, then $\ell(t_iw') = \ell(w') -1$. Hence $s_j t_i \preceq w$.

\end{proof}

\noindent The proof of the following lemma is similar to the proofs of Lemmas \ref{tis3Divw} and \ref{tisjDivw} and is left to the reader.

\begin{lemma}\label{sis(i+1)sisjdivw}

If $s_i \preceq w$ and $s_{i+1} \preceq w$ for $3 \leq i \leq n-1$, then \mbox{$s_is_{i+1}s_i = s_{i+1}s_is_{i+1} \preceq w$}, and if $s_i \preceq w$ and $s_j \preceq w$ for $3 \leq i,j \leq n$ and $|i-j| > 1$, then $s_is_j = s_js_i \preceq w$.

\end{lemma}

The following proposition is a direct consequence of all the preceding lemmas.

\begin{proposition}\label{PropLCMofGenIn1lambdak}

Let $x, y \in X = \{t_0,t_1, \cdots, t_{e-1},s_3, \cdots, s_n \}$. The least common multiple in $([1, \lambda^k], \preceq)$ of $x$ and $y$, denoted by $x \vee y$, exists and is given by the following identities:

\begin{itemize}

\item $t_i \vee t_j = t_kt_0 = t_{i}t_{i-k} = t_{j}t_{j-k}$ for $i \neq j \in \mathbb{Z}/e\mathbb{Z}$,
\item $t_i \vee s_3 = t_is_3t_i = s_3t_is_3$ for $i \in \mathbb{Z}/e\mathbb{Z}$,
\item $t_i \vee s_j = t_is_j = s_jt_i$ for $i \in \mathbb{Z}/e\mathbb{Z}$ and $4 \leq j \leq n$,
\item $s_i \vee s_{i+1} = s_{i}s_{i+1}s_{i} = s_{i+1}s_{i}s_{i+1}$ for $3 \leq i \leq n-1$,
\item $s_i \vee s_j = s_is_j = s_js_i$ for $3 \leq i \neq j \leq n$ and $|i-j| > 1$.

\end{itemize}

\end{proposition}

We have a similar result for $([1,\lambda^k],\preceq_r)$.

\begin{proposition}\label{PropLCMofGenRight}

Let $x, y \in X$. The least common multiple in $([1, \lambda^k], \preceq_r)$ of $x$ and $y$, denoted by $x \vee_r y$, exists and is equal to $x \vee y$.

\end{proposition}

\begin{proof}

Define a map on the generators of $G(e,e,n)$ by $\phi: t_i \longmapsto t_{-i}$ and $s_j \longmapsto s_j$ with $i \in \mathbb{Z}/e\mathbb{Z}$ and $3 \leq j \leq n$. On examination of the relations for $G(e,e,n)$, it is quickly verified that this map extends to an anti-homomorphism $G(e,e,n) \longrightarrow G(e,e,n)$. Clearly $\phi^2$ is the identity, so in particular, $\phi$ respects the length function on $G(e,e,n)$: $\ell(\phi(w)) = \ell(w)$. Suppose that $x,y \in X$ and $x \preceq_r w$ and $y \preceq_r w$, we will show that $x \vee y \preceq_r w$. Write $w = vx$ and $w=v'y$ with $v, v' \in G(e,e,n)$. Thus, $\phi(w) = \phi(x) \phi(v) = \phi(y) \phi(v')$, and since $\phi$ respects length, $\phi(x) \preceq \phi(w)$ and $\phi(y) \preceq \phi(w)$. Hence $\phi(x) \vee \phi(y) \preceq \phi(w)$.

For each pair of generators, it is straightforward to check that\\ $\phi(x) \vee \phi(y) = \phi(x \vee y)$: the only non-trivial case being when $x = t_i$ and $y = t_j$ for $i \neq j$, when we have: $$\phi(t_i) \vee \phi(t_j) = t_{-i} \vee t_{-j} = t_kt_0 = \phi(t_0t_{-k}) = \phi(t_kt_0).$$ Thus, $\phi(x \vee y) \preceq \phi(w)$, that is $\phi(w) = \phi(x \vee y)u$ for some $u \in G(e,e,n)$. Applying $\phi$ again gives $w = \phi(u)(x \vee y)$, and since $\phi$ respects length, $x \vee y \preceq_r w$.

\end{proof}

Note that Propositions \ref{PropLCMofGenIn1lambdak} and \ref{PropLCMofGenRight} are important to prove that both posets $([1,\lambda^k],\preceq)$ and $([1,\lambda^k],\preceq_r)$ are lattices. Actually, they will make possible an induction proof of Proposition \ref{PropInductionExistLCM} in the next subsection.

\subsection{The lattice property and interval structures}

We start by recalling some general properties about lattices that will be useful in our proof. Let $(S, \preceq)$ be a finite poset. 

\begin{definition}\label{DefMeetSemiLattice}

Say that $(S, \preceq)$ is a meet-semilattice if and only if $f \wedge g := gcd(f,g)$ exists for any $f, g \in S$.

\end{definition} 

Equivalently, $(S, \preceq)$ is a meet-semilattice if and only if $\bigwedge T$ exists for any finite nonempty subset $T$ of $S$.

\begin{definition}\label{DefJoinSemiLattice}

Say that $(S, \preceq)$ is a join-semilattice if and only if $f \vee g := lcm(f,g)$ exists for any $f, g \in S$.

\end{definition}

Equivalently, $(S, \preceq)$ is a join-semilattice if and only if $\bigvee T$ exists for any finite nonempty subset $T$ of $S$.

\begin{proposition}\label{PropCaractMeetSemiLattice}

Let $(S, \preceq)$ be a finite poset. $(S, \preceq)$ is a meet-semilattice if and only if for any $f, g \in S$, either $f \vee g$ exists, or $f$ and $g$ have no common multiples.

\end{proposition}

\begin{proof}

Let $f, g \in S$ and suppose that $f$ and $g$ have at least one common multiple. Let $A := \{h \in S\ |\ f \preceq h$ and $g \preceq h \}$ be the set of the common multiples of $f$ and $g$. Since $S$ is finite, $A$ is also finite. Since $(S, \preceq)$ is a meet-semilattice, $\bigwedge A$ exists and $\bigwedge A = lcm(f,g) = f \vee g$.\\
Conversely, let $f, g \in S$ and let $B := \{h \in S\ |\ h \preceq f$ and $h \preceq g \}$ be the set of all common divisors of $f$ and $g$. Since all elements of $B$ have common multiples, $\bigvee B$ exists and we have $\bigvee B = gcd(f,g) = f \wedge g$.

\end{proof}

\begin{definition}\label{DefFromSemiLatToLat}

The poset $(S, \preceq)$ is a lattice if and only if it is both a meet- and join- semilattice.

\end{definition}

The following lemma is a consequence of Proposition \ref{PropCaractMeetSemiLattice}.

\begin{lemma}\label{LemmaSemiLatticeLattice}

If $(S, \preceq)$ is a meet-semilattice such that $\bigvee S$ exists, then $(S, \preceq)$ is a lattice.\\

\end{lemma}

We will prove that $([1,\lambda^k],\preceq)$ is a meet-semilattice by applying Proposition \ref{PropCaractMeetSemiLattice}. For $1 \leq m \leq \ell(\lambda^k)$ with $\ell(\lambda^k) = n(n-1)$, we introduce 

\begin{center} $([1,\lambda^k])_m := \{w \in [1,\lambda^k]\ s.t.\ \ell(w) \leq m \}.$ \end{center}

\begin{proposition}\label{PropInductionExistLCM}

Let $0 \leq k \leq e-1$. For $1 \leq m \leq n(n-1)$ and $u,v$ in $([1,\lambda^k])_m$, either $u \vee v$ exists in $([1,\lambda^k])_m$, or $u$ and $v$ do not have common multiples in $([1,\lambda^k])_m$.

\end{proposition}

\begin{proof}

Let $1 \leq m \leq n(n-1)$. We make a proof by induction on $m$. By Proposition \ref{PropLCMofGenIn1lambdak}, our claim holds for $m=1$. Suppose $m > 1$. Assume that the claim holds for all $1 \leq m' \leq m-1$. We want to prove it for $m' = m$. The proof is illustrated in Figure \ref{fig2} below. Let $u, v$ be in $([1,\lambda^k])_m$ such that $u$ and $v$ have at least one common multiple in $([1,\lambda^k])_m$ which we denote by $w$. Write $u= xu_1$ and $v = y v_1$ such that $x, y \in X$ and $\ell(u) = \ell(u_1) + 1$, $\ell(v) = \ell(v_1) + 1$. By Proposition \ref{PropLCMofGenIn1lambdak}, $x \vee y$ exists. Since $x \preceq w$ and $y \preceq w$, $x \vee y$ divides $w$. We can write $x \vee y = xy_1 = yx_1$ with $\ell(x \vee y) = \ell(x_1) + 1 = \ell(y_1) + 1$. By Lemma \ref{LemmaDivisorsForInduction}, we have $x_1, v_1 \in [1, \lambda^k]$. Also, we have $\ell(x_1) < m$, $\ell(v_1) < m$ and $x_1, v_1$ have a common multiple in $([1,\lambda^k])_{m-1}$. Thus, by the induction hypothesis, $x_1 \vee v_1$ exists in $([1,\lambda^k])_{m-1}$. Similarly, $y_1 \vee u_1$ exists in $([1,\lambda^k])_{m-1}$. Write $x_1 \vee v_1 = v_1x_2 = x_1v_2$ with $\ell(x_1 \vee v_1) = \ell(v_1) + \ell(x_2) = \ell(v_2) + \ell(x_1)$ and write $y_1 \vee u_1 = u_1y_2 = y_1u_2$ with $\ell(y_1 \vee u_1) = \ell(y_1) + \ell(u_2) = \ell(u_1) + \ell(y_2)$. By Lemma \ref{LemmaDivisorsForInduction}, we have $u_2, v_2 \in [1, \lambda^k]$. Also, we have $\ell(u_2) < m$, $\ell(v_2) < m$ and $u_2, v_2$ have a common multiple in $([1,\lambda^k])_{m-1}$. Thus, by the induction hypothesis, $u_2 \vee v_2$ exists in $([1,\lambda^k])_{m-1}$. Write $u_2 \vee v_2 = v_2 u_3 = u_2 v_3$ with $\ell(u_2 \vee v_2) = \ell(v_2) + \ell(u_3) = \ell(u_2) + \ell(v_3)$. Since $uy_2v_3 = vx_2u_3$ is a common multiple of $u$ and $v$ that divides every common multiple $w$ of $u$ and $v$, we deduce that $u \vee v = uy_2v_3 = vx_2u_3$ and we are done.

\end{proof}

\begin{figure}[H]

\begin{small}
\begin{center}

\begin{tikzpicture}[scale=0.7]

\draw [->] (0,0) to node[auto] {$y_2$} (1,0);
\draw [->] (1,0) to node[auto] {$v_3$} (4,0);
\draw [->] (0,2) to node[left] {$u_1$} (0,0);
\draw [->] (1,2) to node[left] {$u_2$} (1,0);
\draw [->] (4,2) to node[left] {$u_3$} (4,0);
\draw [->] (0,2) to node[auto] {$y_1$} (1,2);
\draw [->] (1,2) to node[auto] {$v_2$} (4,2);
\draw [->] (0,3) to node[left] {$x$} (0,2);
\draw [->] (1,3) to node[auto] {$x_1$} (1,2);
\draw [->] (4,3) to node[left] {$x_2$} (4,2);
\draw [->] (0,3) to node[auto] {$y$} (1,3);
\draw [->] (1,3) to node[auto] {$v_1$} (4,3);
\draw [-] (4,0) -- (6,-2) node[right] {$w$};
\draw [->, dashed] (0,0) to[bend right] (6,-2);
\draw [-, dashed] (1,0) to[bend right] (6,-2);
\draw [-, dashed] (1,2) to[bend right] (6,-2);
\draw [-, dashed] (4,2) to[bend left] (6,-2);
\draw [->, dashed] (4,3) to[bend left] (6,-2);

\end{tikzpicture}

\end{center}
\end{small}

\caption{The proof of Proposition \ref{PropInductionExistLCM}.}
\label{fig2}
\end{figure}
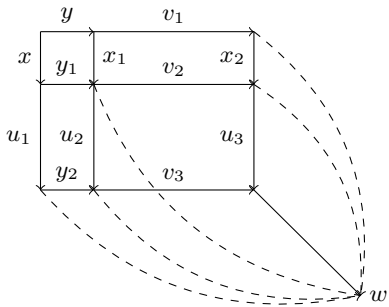

\noindent Similarly, applying Proposition \ref{PropLCMofGenRight}, we obtain the same results for $([1,\lambda^k],\preceq_r)$. We thus proved the following.

\begin{cor}\label{CorBothPosetsLattices}

Both posets $([1,\lambda^k],\preceq)$ and $([1,\lambda^k],\preceq_r)$ are lattices. 

\end{cor}

\begin{proof}

Applying Proposition \ref{PropCaractMeetSemiLattice}, $([1,\lambda^k],\preceq)$ is a meet-semilattice. Also, by definition of the interval $[1,\lambda^k]$, we have $\bigvee [1,\lambda^k] = \lambda^k$. Thus, applying Proposition \ref{PropCaractMeetSemiLattice}, $([1,\lambda^k],\preceq)$ is a lattice. The same can be done for $([1,\lambda^k],\preceq_r)$.

\end{proof}

We are ready to define the interval Garside monoid $M([1,\lambda^k])$.

\begin{definition}\label{DefIntervalMonoid}

Let $\underline{D_k}$ be a set in bijection with $D_k = [1,\lambda^k]$ with $$[1,\lambda^k] \longrightarrow \underline{D_k}: w \longmapsto \underline{w}.$$ We define the monoid $M([1,\lambda^k])$ by the following presentation of monoid with

\begin{itemize}

\item generating set: $\underline{D_k}$ (a copy of the interval $[1,\lambda^k]$) and
\item relations: $\underline{w} = \underline{w'} \hspace{0.1cm} \underline{w''}$ whenever $w,w'$ and $w'' \in [1,\lambda^k]$, $w=w'w''$ and $\ell(w) = \ell(w') + \ell(w'')$.

\end{itemize}

\end{definition}

We have that $\lambda^k$ is balanced. Also, by Corollary \ref{CorBothPosetsLattices}, both posets $([1,\lambda^k],\preceq)$ and $([1,\lambda^k],\preceq_r)$ are lattices. Hence, by Theorem \ref{TheoremMichelGarside}, we have:

\begin{theorem}\label{TheoremIntervalStructure}

$(M([1,\lambda^k]),\underline{\lambda^k})$ is an interval Garside monoid with simples $\underline{D_k}$, where $\underline{D_k}$ is given in Definition \ref{DefIntervalMonoid}. Its group of fractions exists and is denoted by $G(M([1,\lambda^k]))$.

\end{theorem}

Note that these interval structures have been implemented by Michel and the author by using the package CHEVIE of GAP3 (see \cite{CHEVIEJMichel} and \cite{GAP3CPMichelNeaime}). The next two sections are devoted to the study of these interval structures.

\section{Isomorphism with $B(e,e,n)$}\label{SectionIsomB(een)}

We provide a new presentation for the interval Garside monoid $M([1,\lambda^k])$. Furthermore, we prove that when $k$ and $e$ are coprime, meaning that $k\wedge e=1$, the interval Garside group $G(M([1,\lambda^k]))$ is isomorphic to the complex braid group $B(e,e,n)$. Among the results we prove in this section is an important property that is similar to Matsumoto's property in the case of real reflection groups, see Proposition \ref{PropMatsumoto}.

\subsection{Presentations}

Our first aim is to prove that the interval Garside monoid $M([1,\lambda^k])$ is isomorphic to the monoid $B^{\oplus k}(e,e,n)$ defined as follows.

\begin{definition}\label{DefofB+keen}

For $1 \leq k \leq e-1$, we define the monoid $B^{\oplus k}(e,e,n)$ by a monoid presentation with

\begin{itemize}

\item generating set: $\widetilde{X} = \{ \tilde{t}_{0}, \tilde{t}_{1}, \cdots, \tilde{t}_{e-1}, \tilde{s}_{3}, \cdots, \tilde{s}_{n}\}$ and
\item relations: $\left\{
\begin{array}{ll}

\tilde{s}_{i}\tilde{s}_{j}\tilde{s}_{i} = \tilde{s}_{j}\tilde{s}_{i}\tilde{s}_{j} & for\ |i-j|=1,\\
\tilde{s}_{i}\tilde{s}_{j} = \tilde{s}_{j}\tilde{s}_{i} & for\ |i-j| > 1,\\
\tilde{s}_{3}\tilde{t}_{i}\tilde{s}_{3} = \tilde{t}_{i}\tilde{s}_{3}\tilde{t}_{i} & for\ i \in \mathbb{Z}/e\mathbb{Z},\\
\tilde{s}_{j}\tilde{t}_{i} = \tilde{t}_{i}\tilde{s}_{j} & for\ i \in \mathbb{Z}/e\mathbb{Z}\ and\ 4 \leq j \leq n,\\
\tilde{t}_{i}\tilde{t}_{i-k} = \tilde{t}_{j}\tilde{t}_{j-k} & for\ i, j \in \mathbb{Z}/e\mathbb{Z}.

\end{array}
\right.$

\end{itemize}

\end{definition}

Note that the monoid $B^{\oplus 1}(e,e,n)$ is the monoid defined by the presentation of Corran-Picantin considered as a monoid presentation, see Proposition \ref{PropositionPresCPB(een)}. In \cite{CorranPicantin}, this monoid is denoted by $B^{\oplus}(e,e,n)$.\\

Fix $k$ such that $1 \leq k \leq e-1$. We define a diagram for the presentation of $B^{\oplus k}(e,e,n)$ in the same way as the diagram corresponding to the presentation of Corran-Picantin of $B(e,e,n)$ given in \mbox{Figure \ref{PresofCPBeen}}, with a dashed edge between $\tilde{t}_i$ and $\tilde{t}_{i-k}$ and between $\tilde{t}_j$ and $\tilde{t}_{j-k}$ for each relation of the form $\tilde{t}_{i}\tilde{t}_{i-k} = \tilde{t}_{j}\tilde{t}_{j-k}$, $i, j \in \mathbb{Z}/e\mathbb{Z}$. For example, the diagram corresponding to $B^{(2)}(8,8,2)$ is as follows.

\begin{small}
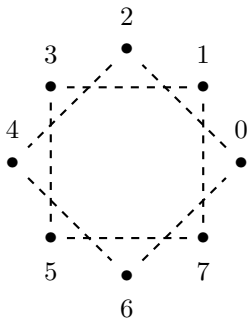
\begin{figure}[H]
\begin{center}
\begin{tikzpicture}

\node[label=$0$] (0) at (1.5,0) {$\bullet$};
\node[label=$1$] (1) at (1,1) {$\bullet$};
\node[label=$2$] (2) at (0,1.5) {$\bullet$};
\node[label=$3$] (3) at (-1,1) {$\bullet$};
\node[label=$4$] (4) at (-1.5,0) {$\bullet$};
\node[label=below:$5$] (5) at (-1,-1) {$\bullet$};
\node[label=below:$6$] (6) at (0,-1.5) {$\bullet$};
\node[label=below:$7$] (7) at (1,-1) {$\bullet$};

\draw[thick,dashed] (0) to (2);
\draw[thick,dashed] (2) to (4);
\draw[thick,dashed] (4) to (6);
\draw[thick,dashed] (6) to (0);

\draw[thick,dashed] (1) to (3);
\draw[thick,dashed] (3) to (5);
\draw[thick,dashed] (5) to (7);
\draw[thick,dashed] (7) to (1);

\end{tikzpicture}
\end{center}
\caption{Diagram for the presentation of $B^{(2)}(8,8,2)$.}\label{PresofB2882}
\end{figure}
\end{small}

The following result is similar to Matsumoto's property in the case of real reflection groups, see \cite{MatsumotoTheorem}.

\begin{proposition}\label{PropMatsumoto}

There exists a map $F: [1,\lambda^k] \longrightarrow B^{\oplus k}(e,e,n)$ defined by\\ $F(w) = \tilde{x}_1 \tilde{x}_2 \cdots \tilde{x}_r$ whenever $\mathbf{x}_1 \mathbf{x}_2 \cdots \mathbf{x}_r$ is a reduced expression over $\mathbf{X}$ of $w$, where $\tilde{x}_i \in \widetilde{X}$ for $1 \leq i \leq r$.

\end{proposition}

\begin{proof}

Let $\mathbf{w}_1$ and $\mathbf{w}_2$ be in $\mathbf{X}^{*}$. We write $\mathbf{w}_1 \overset{\mathcal{B}}{\leadsto} \mathbf{w}_2$ if $\mathbf{w}_2$ is obtained from $\mathbf{w}_1$ by applying only the relations of the presentation of $B^{\oplus k}(e,e,n)$ where we replace $\tilde{t}_i$ by $\mathbf{t}_i$ and $\tilde{s}_j$ by $\mathbf{s}_j$ for all $i \in \mathbb{Z}/e\mathbb{Z}$ and $3 \leq j \leq n$.\\
Let $w$ be in $[1,\lambda^k]$ and suppose that $\mathbf{w}_1$ and $\mathbf{w}_2$ are two reduced expressions over $\mathbf{X}$ of $w$. We prove that $\mathbf{w}_1 \overset{\mathcal{B}}{\leadsto} \mathbf{w}_2$ by induction on $\boldsymbol{\ell}(\mathbf{w}_1)$.\\
The result holds vacuously for $\boldsymbol{\ell}(\mathbf{w}_1) = 0$ and $\boldsymbol{\ell}(\mathbf{w}_1) = 1$. Suppose that $\boldsymbol{\ell}(\mathbf{w}_1) > 1$. Write 
\begin{center} $\mathbf{w}_1 = \mathbf{x}_1\mathbf{w'}_1$ and $\mathbf{w}_2 = \mathbf{x}_2\mathbf{w'}_2$, with $\mathbf{x}_1, \mathbf{x}_2 \in \mathbf{X}$.\end{center}
\underline{If $\mathbf{x}_1 = \mathbf{x}_2$}, we have $x_1 w'_{1} = x_2 w'_{2}$ in $G(e,e,n)$ from which we get $w'_{1} = w'_{2}$. Then, by the induction hypothesis, we have $\mathbf{w'}_1 \overset{\mathcal{B}}{\leadsto} \mathbf{w'}_2$. Hence $\mathbf{w}_1 \overset{\mathcal{B}}{\leadsto} \mathbf{w}_2$.\\
\underline{If $\mathbf{x}_1 \neq \mathbf{x}_2$}, since $x_1 \preceq w$ and $x_2 \preceq w$, we have $x_1 \vee x_2 \preceq w$ where $x_1 \vee x_2$ is the lcm of $x_1$ and $x_2$ in $([1,\lambda^k],\preceq)$ given in Proposition \ref{PropLCMofGenIn1lambdak}. Write $w = (x_1 \vee x_2) w'$. Also, write $\mathbf{x}_1 \vee \mathbf{x}_2 = \mathbf{x}_1 \mathbf{v}_1$ and $\mathbf{x}_1 \vee \mathbf{x}_2 = \mathbf{x}_2 \mathbf{v}_2$ where we can check that $\mathbf{x}_1 \mathbf{v}_1 \overset{\mathcal{B}}{\leadsto} \mathbf{x}_2 \mathbf{v}_2$ for all possible cases for $\mathbf{x}_1$ and $\mathbf{x}_2$. All the words $\mathbf{x}_1\mathbf{w'}_1$, $\mathbf{x}_2\mathbf{w'}_2$, $\mathbf{x}_1\mathbf{v}_1\mathbf{w'}$, and $\mathbf{x}_2\mathbf{v}_2\mathbf{w'}$ represent $w$. In particular, $\mathbf{x}_1\mathbf{w'}_1$ and $\mathbf{x}_1\mathbf{v}_1\mathbf{w'}$ represent $w$. Hence $w'_1 = v_1w'$ and by the induction hypothesis, we have $\mathbf{w'}_1 \overset{\mathcal{B}}{\leadsto} \mathbf{v}_1\mathbf{w'}$. Thus, we have
\begin{center} $\mathbf{x}_1\mathbf{w'}_1 \overset{\mathcal{B}}{\leadsto} \mathbf{x}_1\mathbf{v}_1\mathbf{w'}$. \end{center}
Similarly, since $\mathbf{x}_2\mathbf{w'}_2$ and $\mathbf{x}_2\mathbf{v}_2\mathbf{w'}$ represent $w$, we get
\begin{center} $\mathbf{x}_2\mathbf{v}_2\mathbf{w'} \overset{\mathcal{B}}{\leadsto} \mathbf{x}_2\mathbf{w'}_2$. \end{center}
Since $\mathbf{x}_1 \mathbf{v}_1 \overset{\mathcal{B}}{\leadsto} \mathbf{x}_2 \mathbf{v}_2$, we have
\begin{center} $\mathbf{x}_1\mathbf{v}_1\mathbf{w'} \overset{\mathcal{B}}{\leadsto} \mathbf{x}_2\mathbf{v}_2\mathbf{w'}$. \end{center} We obtain:
\begin{center}$\mathbf{w}_1 \overset{\mathcal{B}}{\leadsto} \mathbf{x}_1\mathbf{w'}_1 \overset{\mathcal{B}}{\leadsto}  \mathbf{x}_1\mathbf{v}_1\mathbf{w'} \overset{\mathcal{B}}{\leadsto} \mathbf{x}_2\mathbf{v}_2\mathbf{w'} \overset{\mathcal{B}}{\leadsto} \mathbf{x}_2\mathbf{w'}_2 \overset{\mathcal{B}}{\leadsto} \mathbf{w}_2$.\end{center}
Hence $\mathbf{w}_1 \overset{\mathcal{B}}{\leadsto} \mathbf{w}_2$ and we are done.

\end{proof}

\begin{remark}

In the case of the Coxeter group $G(2,2,n)$ of type $D_n$, we have $[1,\lambda] = G(2,2,n)$, where $\lambda$ is the unique balanced element of maximal length of $G(2,2,n)$, see Remarks \ref{RemUniqueMaximalG(2,2,n)} and \ref{RemDivisorsLambdaG(2,2,n)}. Hence the previous result is valid for $G(2,2,n)$. This is Matsumoto's property for the case of $G(2,2,n)$, see \cite{MatsumotoTheorem}.

\end{remark}

By the following proposition, we provide an alternative presentation for the interval Garside monoid $M([1,\lambda^k])$ given in Definition \ref{DefIntervalMonoid}.

\begin{proposition}

The monoid $B^{\oplus k}(e,e,n)$ is isomorphic to $M([1,\lambda^k])$.

\end{proposition}

\begin{proof}

Consider the map $\rho: \underline{D_k} \longrightarrow B^{\oplus k}(e,e,n): \underline{w} \longmapsto F(w)$ where $F$ is defined in Proposition \ref{PropMatsumoto}.
Let $\underline{w} = \underline{w}' \underline{w}''$ be a defining relation of $M([1,\lambda^k])$. Since $\ell(w) = \ell(w') + \ell(w'')$, a reduced expression for $w'w''$ is obtained by concatenating reduced expressions for $w'$ and $w''$. It follows that $F(w'w'') = F(w') F(w'')$. We conclude that $\rho$ has a unique extension to a monoid homomorphism $M([1,\lambda^k]) \longrightarrow B^{\oplus k}(e,e,n)$, which we denote by the same symbol.

Conversely, consider the map $\rho' : \widetilde{X} \longrightarrow M([1,\lambda^k]): \widetilde{x} \longmapsto \underline{x}$. In order to prove that $\rho'$ extends to a unique monoid homomorphism $B^{\oplus k}(e,e,n) \longrightarrow M([1,\lambda^k])$, we have to check that $\underline{w_1} = \underline{w_2}$ in $M([1,\lambda^k])$ for any defining relation $\widetilde{w}_1 = \widetilde{w}_2$ of $B^{\oplus k}(e,e,n)$. Given a relation $\widetilde{w}_1 = \widetilde{w}_2 = \tilde{x}_1\tilde{x}_2 \cdots \tilde{x}_r$ of $B^{\oplus k}(e,e,n)$, we have $\boldsymbol{w_1} = \boldsymbol{w_2} = \mathbf{x}_1 \mathbf{x}_2 \cdots \mathbf{x}_r$ a reduced word over $\mathbf{X}$. On the other hand, applying repeatedly the defining relations in $M([1,\lambda^k])$ yields to $\underline{w} = \underline{x_1}\ \underline{x_2} \cdots \underline{x_r}$ if $\boldsymbol{w} = \mathbf{x}_1\mathbf{x}_2 \cdots \mathbf{x}_r$ is a reduced expression over $\mathbf{X}$. Thus, we can conclude that $\underline{w_1} = \underline{w_2}$, as desired.

Hence we have defined two homomorphisms $\rho: \underline{D_k} \longrightarrow B^{\oplus k}(e,e,n)$ and $\rho': \widetilde{X} \longrightarrow M([1,\lambda^k])$ such that $\rho \circ \rho' = id_{B^{\oplus k}(e,e,n)}$ and $\rho' \circ \rho = id_{M([1,\lambda^k])}$. It follows that $B^{\oplus k}(e,e,n)$ is isomorphic to $M([1,\lambda^k])$.

\end{proof}

Since $B^{\oplus k}(e,e,n)$ is isomorphic to $M([1,\lambda^k])$, we deduce that $B^{\oplus k}(e,e,n)$ is a Garside monoid and we denote by $B^{(k)}(e,e,n)$ its group of fractions.

\subsection{Identifying $B(e,e,n)$}

Now, we want to check which of the monoids $B^{\oplus k}(e,e,n)$ are isomorphic to $B^{\oplus}(e,e,n)$. Assume there exists an isomorphism $\phi : B^{\oplus k}(e,e,n) \longrightarrow  B^{\oplus}(e,e,n)$ for a given $k$ with $1 \leq k \leq e-1$. We start with the following lemma.

\begin{lemma}\label{LemmaIsomMonoids}

The isomorphism $\phi$ fixes $\tilde{s}_3, \tilde{s}_4, \cdots, \tilde{s}_n$ and permutes the $\tilde{t}_i$ where \mbox{$i \in \mathbb{Z}/e\mathbb{Z}$}.  

\end{lemma}

\begin{proof}

Let $f$ be in $\widetilde{X}^{*}$. We have $\boldsymbol{\ell}(f) \leq \boldsymbol{\ell}(\phi(f))$. Thus, we have $\boldsymbol{\ell}(\tilde{x}) \leq \boldsymbol{\ell}(\phi(\tilde{x}))$ for $\tilde{x} \in \widetilde{X}$. Also, $\boldsymbol{\ell}(\phi(\tilde{x})) \leq \boldsymbol{\ell}(\phi^{-1}(\phi(\tilde{x}))) = \boldsymbol{\ell}(x)$. Hence $\boldsymbol{\ell}(\tilde{x}) = \boldsymbol{\ell}(\phi(\tilde{x})) = 1$. It follows that $\phi$ maps generators to generators, that is $\phi(\tilde{x}) \in \{\tilde{t}_0, \tilde{t}_1, \cdots, \tilde{t}_{e-1},\tilde{s}_3, \cdots, \tilde{s}_n\}$.

Furthermore, the only generator of  $B^{\oplus k}(e,e,n)$ that commutes with all other generators except for one of them is $\tilde{s}_{n}$. On the other hand, $\tilde{s}_n$ is the only generator of $B^{\oplus}(e,e,n)$ that satisfies the latter property. Hence $\phi(\tilde{s}_{n}) = \tilde{s}_n$. Next, $\tilde{s}_{n-1}$ is the only generator of $B^{\oplus k}(e,e,n)$ that does not commute with $\tilde{s}_{n}$. The only generator of $B^{\oplus}(e,e,n)$ that does not commute with $\tilde{s}_n$ is also $\tilde{s}_{n-1}$. Hence $\phi(\tilde{s}_{n-1}) = \tilde{s}_{n-1}$. Next, the only generator of  $B^{\oplus k}(e,e,n)$ different from $\tilde{s}_{n}$ and that does not commute with $\tilde{s}_{n-1}$ is $\tilde{s}_{n-2}$. And so on, we get $\phi(\tilde{s}_{j}) = \tilde{s}_j$ for $3 \leq j \leq n$. It remains that $\phi(\{ \tilde{t}_{i}\ |\ 0 \leq i \leq e-1 \}) = \{ \tilde{t}_i\ |\ 0 \leq i \leq e-1 \}$.

\end{proof}

\begin{proposition}\label{PropIsomwithMonoidCP}

The monoids $B^{\oplus k}(e,e,n)$ and $B^{\oplus}(e,e,n)$ are isomorphic if and only if $k \wedge e = 1$.

\end{proposition}

\begin{proof}

Assume there exists an isomorphism $\phi$ between the monoids $B^{\oplus k}(e,e,n)$ and $B^{\oplus}(e,e,n)$. By Lemma \ref{LemmaIsomMonoids}, we have $\phi(\tilde{s}_{j}) = \tilde{s}_j$ for $3 \leq j \leq n$ and \mbox{$\phi(\{ \tilde{t}_{i}\ |\ 0 \leq i \leq e-1 \})$} $= \{ \widetilde{t}_i\ |\ 0 \leq i \leq e-1 \}$. Write $\tilde{t}_{a_i}$ for the generator of $B^{\oplus k}(e,e,n)$ for which $\phi(\tilde{t}_{a_i}) = \tilde{t}_i$. The existence of the isomorphism implies $\tilde{t}_{a_{i+1}}\tilde{t}_{a_i} = \tilde{t}_{a_1}\tilde{t}_{a_0}$ for each $i \in \mathbb{Z}/e\mathbb{Z}$. But the only relations of this type in $B^{\oplus k}(e,e,n)$ are $\tilde{t}_{a_{i}+k}\tilde{t}_{a_i} = \tilde{t}_{a_0 + k}\tilde{t}_{a_0}$. Thus, $a_1 = a_0+k$, $a_2 = a_0+2k$, $\cdots$, $a_{e-1} = a_0 + (e-1)k$. Since $\phi$ is bijective on the $\tilde{t}_i$-type generators, $\left\{ a_0 + ik\ |\ i \in \mathbb{Z}/e\mathbb{Z} \right\} = \mathbb{Z}/e\mathbb{Z}$, so $k \wedge e = 1$.

Conversely, let $1 \leq k \leq e-1$ such that $k \wedge e = 1$. We define a map $\phi : B^{\oplus}(e,e,n) \longrightarrow B^{\oplus k}(e,e,n)$ where $\phi(\tilde{s}_j) = \tilde{s}_{j}$ for $3 \leq j \leq n$ and $\phi(\tilde{t}_i) = \phi(\tilde{t}_{ik})$, that is $\phi(\tilde{t}_0) = \tilde{t}_{0}$, $\phi(\tilde{t}_1) = \tilde{t}_{k}$, $\phi(\tilde{t}_2) = \tilde{t}_{2k}$, $\cdots$, $\phi(\tilde{t}_{e-1}) = \widetilde{t}_{(e-1)k}$. The map $\phi$ is a well-defined monoid homomorphism, which is both surjective (as it sends a generator of $B^{\oplus}(e,e,n)$ to a generator of $B^{\oplus k}(e,e,n)$) and injective (as it is bijective on the relations). Hence $\phi$ defines an isomorphism of monoids.

\end{proof}

When $k \wedge e = 1$, since $B^{\oplus k}(e,e,n)$ is isomorphic to $B^{\oplus}(e,e,n)$, we have the following.

\begin{cor}\label{CorIsomGrFractBraidGroup}

$B^{(k)}(e,e,n)$ is isomorphic to the complex braid group $B(e,e,n)$ for $k \wedge e = 1$.

\end{cor}

The reason that the proof of Proposition \ref{PropIsomwithMonoidCP} fails in the case $k \wedge e \neq 1$ is that we have more than one connected component in $\Gamma_k$ that link $\tilde{t}_0$, $\tilde{t}_1$, $\cdots$, and $\tilde{t}_{e-1}$ together, as we can see in Figure \ref{PresofB2882}. Actually, it is easy to check that the number of connected components that link $\tilde{t}_0$, $\tilde{t}_1$, $\cdots$, and $\tilde{t}_{e-1}$ together is the number of cosets of the subgroup of $\mathbb{Z}/e\mathbb{Z}$ generated by the class of $k$, that is equal to $k \wedge e$, and each of these cosets have $e'=e/k \wedge e$ elements. This will be useful in the next section.

\section{New Garside structures}\label{SectionNewGarsideGroups}

When $k \wedge e \neq 1$, we describe $B^{(k)}(e,e,n)$ as an amalgamated product of $k \wedge e$ copies of the complex braid group $B(e',e',n)$ with $e'=e/e \wedge k$, over a common subgroup which is the Artin-Tits group $B(2,1,n-1)$. This allows us to compute the center of $B^{(k)}(e,e,n)$. Finally, using the Garside structure of $B^{(k)}(e,e,n)$, we compute its first and second integral homology groups using the Dehornoy-Lafont complex \cite{DehLafHomologyofGaussianGroups} and the method used in \cite{CalMarHomologyComputations}.\\

By an adaptation of the results of Crisp \cite{CrispInjectivemaps} as in Lemma 5.2 of \cite{CalMarHomologyComputations}, we have the following embedding. Let $B:=B(2,1,n-1)$ be the Artin-Tits group defined by a presentation with generators $q_1$, $q_2$, $\cdots$, $q_{n-1}$ and relations that can be described by the following diagram presentation. This diagram follows the standard conventions for Artin-Tits diagrams. 

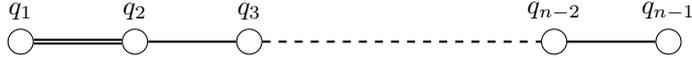
\begin{figure}[H]
\begin{center}
\begin{tikzpicture}

\node[draw, shape=circle, label=above:$q_1$] (1) at (0,0) {};
\node[draw, shape=circle, label=above:$q_2$] (2) at (1.5,0) {};
\node[draw, shape=circle, label=above:$q_3$] (3) at (3,0) {};
\node[draw, shape=circle, label=above:$q_{n-2}$] (n-2) at (7,0) {};
\node[draw, shape=circle, label=above:$q_{n-1}$] (n-1) at (8.5,0) {};

\draw[double,thick,-] (1) to (2);
\draw[thick,-] (2) to (3);
\draw[thick,-] (n-2) to (n-1);
\draw[thick,dashed] (3) to (n-2);

\end{tikzpicture}
\end{center}
\end{figure}

\begin{proposition}

The group $B$ injects in $B^{(k)}(e,e,n)$.

\end{proposition}

\begin{proof}

Define a monoid homomorphism $\phi : B^{+} \longrightarrow B^{\oplus k}(e,e,n) : q_{1} \longmapsto \tilde{t}_i\tilde{t}_{i-k}$, \mbox{$q_2 \longmapsto \tilde{s}_3$}, $\cdots$, $q_{n-1} \longmapsto \tilde{s}_n$. It is easy to check that for all $x, y \in \{q_1, q_2, \cdots, q_{n-1}\}$, we have $lcm(\phi(x), \phi(y))=\phi(lcm(x,y))$. Hence by applying Lemma 5.2 of \cite{CalMarHomologyComputations}, \mbox{$B(2,1,n-1)$} injects in $B^{(k)}(e,e,n)$.

\end{proof}

We construct $B^{(k)}(e,e,n)$ as follows.

\begin{proposition}\label{PropBkeenAmalgamatedProduct}

Let $B(1) := B(e',e',n)$ where $e'=e/e \wedge k$. Inductively, define $B(i+1)$ to be the amalgamated product $B(i+1) := B(i) *_{B} B(e',e',n)$ over $B = B(2,1,n-1)$. Then we have $B^{(k)}(e,e,n)$ is isomorphic to $B(e \wedge k)$.

\end{proposition}

\begin{proof}

Due to the presentation of $B^{(k)}(e,e,n)$ given in Definition \ref{DefofB+keen} and to the presentation of the amalgamated products (see Section 4.2 of \cite{CombinatorialGroupTheory}), one can deduce that $B(e \wedge k)$ is isomorphic to $B^{(k)}(e,e,n)$.

\end{proof}

\begin{figure}[H]
\begin{center}
\begin{tikzpicture}

\node[draw, shape=rectangle, label=below:{$B(e',e',n)$}] (1) at (0,0) {};
\node[draw, shape=rectangle, label=below:{$B(e',e',n)$}] (2) at (2,0) {};
\node[draw, shape=rectangle, label=below:{$B(e',e',n)$}] (3) at (-1,1) {};
\node[draw, shape=rectangle, label=right:{$B(2)$}] (4) at (1,1) {};
\node[draw, shape=rectangle, label=right:{$B(3)$}] (5) at (0,2) {};
\node[draw, shape=rectangle, label=left:{$B(e',e',n)$}] (6) at (-3,3) {};
\node[draw, shape=rectangle] (7) at (-1,3) {};
\node[draw, shape=rectangle, label=right:{$B(e \wedge k) \simeq B^{(k)}(e,e,n)$}] (8) at (-2,4) {};

\draw[thick,-] (1) to (4);
\draw[thick,-] (2) to (4);
\draw[thick,-] (5) to (3);
\draw[thick,-] (5) to (4);
\draw[thick,dashed,-] (7) to (5);
\draw[thick,dashed,-] (-1.5,1.5) to (-2.5,2.5);
\draw[thick,-] (8) to (6);
\draw[thick,-] (8) to (7);

\end{tikzpicture}
\end{center}
\caption{The construction of $B^{(k)}(e,e,n)$.}
\end{figure}
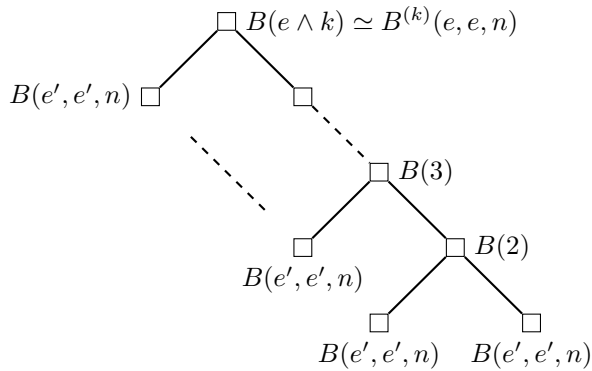

\begin{example}

Consider the case of $B^{(2)}(6,6,3)$. It is an amalgamated product of $k \wedge e = 2$ copies of $B(e',e',3)$ with $e' = e/(k \wedge e) = 3$ over the Artin-Tits group $B(2,1,2)$. Consider the following diagram of this amalgamation.\\
The presentation of $B(3,3,3) * B(3,3,3)$ over $B(2,1,2)$ is as follows:
\begin{itemize}

\item the generators are the union of the generators of the two copies of $B(3,3,3)$, 
\item the relations are the union of the relations of the two copies of $B(3,3,3)$ with the additional relations $\tilde{s}_3 = \tilde{s}'_3$ and $\tilde{t}_2\tilde{t}_0 = \tilde{t}_3\tilde{t}_1$

\end{itemize}
This is exactly the presentation of $B^{(2)}(6,6,3)$ given in Definition \ref{DefofB+keen}.

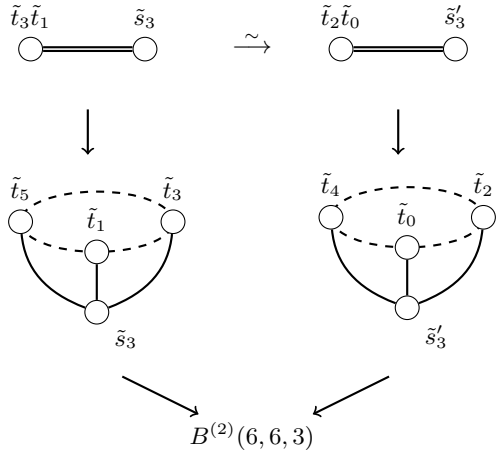
\begin{small}
\begin{center}

\begin{tabular}{lll}

\begin{tikzpicture}
\node[draw, shape=circle, label=above:$\tilde{t}_3\tilde{t}_1$] (1) at (0,0) {};
\node[draw, shape=circle, label=above:$\tilde{s}_3$] (2) at (1.5,0) {};
\draw[double,thick,-] (1) to (2);
\draw[thick,->] (0.75,-0.8) to (0.75,-1.4);
\end{tikzpicture} & \begin{tikzpicture}
\node[] () at (0,0) {};
\node[] () at (0,1.4) {$\overset{\sim}{\longrightarrow}$};
\end{tikzpicture}& \begin{tikzpicture}
\node[draw, shape=circle, label=above:$\tilde{t}_2\tilde{t}_0$] (1) at (0,0) {};
\node[draw, shape=circle, label=above:$\tilde{s}'_3$] (2) at (1.5,0) {};
\draw[double,thick,-] (1) to (2);
\draw[thick,->] (0.75,-0.8) to (0.75,-1.4);
\end{tikzpicture}\\

\begin{tikzpicture}[yscale=0.8,xscale=1]

\draw[thick,dashed] (0,0) ellipse (1cm and 0.5cm);

\node[draw, shape=circle, fill=white, label=above:$\tilde{t}_1$] (t1) at (0,-0.5) {};
\node[draw, shape=circle, fill=white, label=above:$\tilde{t}_3$] (t3) at (1,0) {};
\node[draw, shape=circle, fill=white, label=above:$\tilde{t}_5$] (t5) at (-1,0) {};

\node[draw, shape=circle, fill=white, label=below right:$\tilde{s}_3$] (s3) at (0,-1.5) {};

\draw[thick,-] (t1) to (s3);
\draw[thick,-, bend left] (t3) to (s3);
\draw[thick,-,bend right] (t5) to (s3);

\end{tikzpicture} & & \begin{tikzpicture}[yscale=0.8,xscale=1]

\draw[thick,dashed] (0,0) ellipse (1cm and 0.5cm);

\node[draw, shape=circle, fill=white, label=above:$\tilde{t}_0$] (t0) at (0,-0.5) {};
\node[draw, shape=circle, fill=white, label=above:$\tilde{t}_2$] (t2) at (1,0) {};
\node[draw, shape=circle, fill=white, label=above:$\tilde{t}_4$] (t4) at (-1,0) {};

\node[draw, shape=circle, fill=white, label=below right:$\tilde{s}'_3$] (s3) at (0,-1.5) {};

\draw[thick,-] (t0) to (s3);
\draw[thick,-, bend left] (t2) to (s3);
\draw[thick,-,bend right] (t4) to (s3);

\end{tikzpicture}\\

\end{tabular}

\begin{tikzpicture}

\node[] () at (0,0) {};
\draw[thick,->] (0,0) to (1,-0.5);
\draw[thick,->] (3.5,0) to (2.5,-0.5);

\end{tikzpicture}

$B^{(2)}(6,6,3)$

\end{center}
\end{small}

\end{example}

\begin{proposition}

The center of $B^{(k)}(e,e,n)$ is infinite cyclic isomorphic to $\mathbb{Z}$.

\end{proposition}

\begin{proof}

By Corollary 4.5 of \cite{CombinatorialGroupTheory} that computes the center of an amalgamated product, the center of $B^{(k)}(e,e,n)$ is the intersection of the centers of $B$ and $B(e',e',n)$. Since the center of $B$ and $B(e',e',n)$ is infinite cyclic \cite{BMR}, the center of $B^{(k)}(e,e,n)$ is infinite cyclic isomorphic to $\mathbb{Z}$.

\end{proof}

Since the center of $B(e,e,n)$ is also isomorphic to $\mathbb{Z}$ (see \cite{BMR}), in order to distinguish $B^{(k)}(e,e,n)$ from the braid groups $B(e,e,n)$, we compute its first and second integral homology groups. We recall the Dehornoy-Lafont complex and follow the method in \cite{CalMarHomologyComputations} where the second integral homology group of $B(e,e,n)$ is computed.\\

We order the elements of $\widetilde{X}$ by considering $\tilde{s}_{n} < \tilde{s}_{n-1} < \cdots < \tilde{s}_{3} < \tilde{t}_{0} < \tilde{t}_{1} < \cdots < \tilde{t}_{e-1}$. For $f \in B^{\oplus k}(e,e,n)$, denote by $d(f)$ the least element in $\widetilde{X}$ which divides $f$ on the right. An $r$-cell is an $r$-tuple $[x_1, \cdots, x_r]$ of elements in $\widetilde{X}$ such that $x_1 < x_2 < \cdots < x_r$ and $x_i = d(lcm(x_i,x_{i+1}, \cdots, x_{r}))$. The set $C_r$ of $r$-chains is the free $\mathbb{Z} B^{\oplus k}(e,e,n)$-module with basis $\widetilde{X}_r$, the set of all $r$-cells with the convention $\widetilde{X}_0 = \{[\emptyset]\}$. We provide the definition of the differential $\partial_r : C_r \longrightarrow C_{r-1}$.

\begin{definition}\label{DefintionDifferentialHomology}

Let $[\alpha,A]$ be an $(r+1)$-cell, with $\alpha \in \widetilde{X}$ and $A$ an $r$-cell. Denote $\alpha_{/A}$ the unique element of $B^{\oplus k}(e,e,n)$ such that $(\alpha_{/A})lcm(A) = lcm(\alpha,A)$. Define the differential $\partial_r : C_r \longrightarrow C_{r-1}$ recursively through two $\mathbb{Z}$-module homomorphisms $s_r: C_r \longrightarrow C_{r+1}$ and $u_r: C_r \longrightarrow C_r$ as follows.

\begin{center}

$\partial_{r+1}[\alpha,A] = \alpha_{/A}[A] - u_r(\alpha_{/A}[A])$,

\end{center}
with $u_{r+1} = s_r \circ \partial_{r+1}$ where $u_0(f[\emptyset]) = [\emptyset]$, for all $f \in B^{\oplus k}(e,e,n)$, and\\
$s_r([\emptyset])=0$, $s_r(x[A]) = 0$ if $\alpha:=d(x lcm(A))$ coincides with the first coefficient in $A$, and otherwise
$s_r(x[A]) = y[\alpha,A] + s_{r}(y u_{r}(\alpha_{/A}[A]))$ with $x = y \alpha_{/A}$.\\

\end{definition}

We provide the result of the computation of $\partial_1$, $\partial_2$, and $\partial_3$ for all $1$, $2$, and $3$-cells, respectively.
For all $x \in \widetilde{X}$, we have
\begin{center}

$\partial_1[x] = (x-1)[\emptyset]$,

\end{center}
for all $1 \leq i \leq e-1$,
\begin{center}

$\partial_2[\tilde{t}_0,\tilde{t}_i] = \tilde{t}_{i+k}[\tilde{t}_i] - \tilde{t}_k[\tilde{t}_0] - [\tilde{t}_k] + [\tilde{t}_{i+k}]$,

\end{center}
for $x,y \in \widetilde{X}$ with $xyx=yxy$,
\begin{center}

$\partial_2[x,y] = (yx+1-x)[y] + (y-xy-1)[x]$, and

\end{center}
for $x,y \in \widetilde{X}$ with $xy=yx$,
\begin{center}

$\partial_2[x,y] = (x-1)[y] - (y-1)[x]$.

\end{center}
For $j \neq -k$ mod $e$, we have:
\begin{center}

$\partial_3[\tilde{s}_3,\tilde{t}_0,\tilde{t}_j] = (\tilde{s}_3\tilde{t}_k\tilde{t}_0\tilde{s}_3 - \tilde{t}_k\tilde{t}_0\tilde{s}_3 + \tilde{t}_{j+2k}\tilde{s}_3)[\tilde{t}_0,\tilde{t}_j] - \tilde{t}_{j+2k}\tilde{s}_3\tilde{t}_{j+k}[\tilde{s}_3,\tilde{t}_j] + (\tilde{t}_{j+2k} - \tilde{s}_3\tilde{t}_{j+2k})[\tilde{s}_3,\tilde{t}_{j+k}] + (\tilde{s}_3 - \tilde{t}_{j+2k}\tilde{s}_3 - 1)[\tilde{t}_0, \tilde{t}_{j+k}] + (\tilde{s}_3 \tilde{t}_{2k} - \tilde{t}_{2k})[\tilde{s}_3,\tilde{t}_k] + (\tilde{t}_{2k}\tilde{s}_3 + 1 - \tilde{s}_3)[\tilde{t}_0,\tilde{t}_k] + [\tilde{s}_3,\tilde{t}_{j+2k}] + \tilde{t}_{2k}\tilde{s}_3\tilde{t}_k[\tilde{s}_3,\tilde{t}_0] - [\tilde{s}_3, \tilde{t}_{2k}]$ and

\end{center}

\begin{center}

$\partial_3[\tilde{s}_3,\tilde{t}_0,\tilde{t}_{-k}] = (\tilde{s}_3\tilde{t}_k\tilde{t}_0\tilde{s}_3 - \tilde{t}_k\tilde{t}_0\tilde{s}_3 + \tilde{t}_k\tilde{s}_3)[\tilde{t}_0, \tilde{t}_{-k}] - \tilde{t}_k\tilde{s}_3\tilde{t}_0[\tilde{s}_3,\tilde{t}_{-k}] + (1-\tilde{t}_{2k}+\tilde{s}_3\tilde{t}_{2k})[\tilde{s}_3,\tilde{t}_k] + (1+\tilde{t}_{2k}\tilde{s}_3 - \tilde{s}_3)[\tilde{t}_0,\tilde{t}_k] + (\tilde{t}_k - \tilde{s}_3\tilde{t}_k + \tilde{t}_{2k}\tilde{s}_3\tilde{t}_k)[\tilde{s}_3,\tilde{t}_0] - [\tilde{s}_3,\tilde{t}_{2k}]$.

\end{center}
Also, for $1 \leq i \leq e-1$ and $4 \leq j \leq n$, we have:
\begin{center}

$\partial_3[\tilde{s}_j,\tilde{t}_0,\tilde{t}_i] = (\tilde{s}_j -1)[\tilde{t}_0,\tilde{t}_i] - \tilde{t}_{i+k}[\tilde{s}_j, \tilde{t}_i] + \tilde{t}_k[\tilde{s}_j,\tilde{t}_0] - [\tilde{s}_j,\tilde{t}_{i+k}] + [\tilde{s}_j,\tilde{t}_k]$,

\end{center}
for $x,y,z \in \widetilde{X}$ with $xyx=yxy$, $xz=zx$, and $yzy=zyz$,
\begin{center}

$\partial_3[x,y,z] = (z + xyz - yz -1)[x,y] - [x,z] + (xz-z-x+1-yxz)y[x,z] + (x-1-yx+zyx)[y,z]$,

\end{center}
for $x,y,z \in \widetilde{X}$ with $xyx=yxy$, $xz=zx$, and $yz=zy$,
\begin{center}

$\partial_3[x,y,z] = (1-x+yx)[y,z] + (y-1-xy)[x,z] + (z-1)[x,y]$,

\end{center}
for $x,y,z \in \widetilde{X}$ with $xy=yx$, $xz=zx$, and $yzy=zyz$,
\begin{center}

$\partial_3[x,y,z] = (1+yz-z)[x,y] + (y-1-zy)[x,z] + (x-1)[y,z]$, and

\end{center}
for $x,y,z \in \widetilde{X}$ with $xy=yx$, $xz=zx$, and $yz=zy$,
\begin{center}

$\partial_3[x,y,z] = (1-y)[x,z] + (z-1)[x,y] +(x-1)[y,z]$.

\end{center}

Let $d_r = \partial_r \otimes_{\mathbb{Z} B^{\oplus k}(e,e,n)} \mathbb{Z} : C_r  \otimes_{\mathbb{Z} B^{\oplus k}(e,e,n)} \mathbb{Z} \longrightarrow C_{r-1}  \otimes_{\mathbb{Z} B^{\oplus k}(e,e,n)} \mathbb{Z}$ be the differential with trivial coefficients. For example, for $d_2$, we have:
for all $1 \leq i \leq e-1$,
\begin{center}

$d_2[\tilde{t}_0,\tilde{t}_i] = [\tilde{t}_i] - [\tilde{t}_0] - [\tilde{t}_k] + [\tilde{t}_{i+k}]$,

\end{center}
for $x,y \in \widetilde{X}$ with $xyx=yxy$,
\begin{center}

$d_2[x,y] = [y] - [x]$, and

\end{center}
for $x,y \in \widetilde{X}$ with $xy=yx$,
\begin{center}

$d_2[x,y] = 0$.

\end{center}
The same can be done for $d_3$.

\begin{definition}

Define the integral homology group of order $r$ to be $$H_r(B^{(k)}(e,e,n),\mathbb{Z}) = ker(d_r) / Im(d_{r+1}).$$

\end{definition}

We are ready to compute the first and second integral homology groups of $B^{(k)}(e,e,n)$. Using the presentation of $B^{(k)}(e,e,n)$ given in Definition \ref{DefofB+keen}, one can check that the abelianization of $B^{(k)}(e,e,n)$ is isomorphic to $\mathbb{Z}$. Since $H_1(B^{(k)}(e,e,n),\mathbb{Z})$ is isomorphic to the abelianization of $B^{(k)}(e,e,n)$, we deduce that $H_1(B^{(k)}(e,e,n),\mathbb{Z})$ is isomorphic to $\mathbb{Z}$. Since $H_1(B(e,e,n),\mathbb{Z})$ is also isomorphic to $\mathbb{Z}$ (see \cite{BMR}), the first integral homology group does not give any additional information whether these groups are isomorphic to a certain complex braid group of type $B(e,e,n)$ or not.\\

Recall that by \cite{CalMarHomologyComputations}, we have

\begin{itemize}

\item $H_2(B(e,e,3),\mathbb{Z}) \simeq \mathbb{Z}/e\mathbb{Z}$ where $e \geq 2$,
\item $H_2(B(e,e,4),\mathbb{Z}) \simeq \mathbb{Z}/e\mathbb{Z} \times \mathbb{Z}/2\mathbb{Z}$ when $e$ is odd and $H_2(B(e,e,4),\mathbb{Z}) \simeq \mathbb{Z}/e\mathbb{Z} \times (\mathbb{Z}/2\mathbb{Z})^2$ when $e$ is even,
\item $H_2(B(e,e,n),\mathbb{Z}) \simeq \mathbb{Z}/e\mathbb{Z} \times \mathbb{Z}/2\mathbb{Z}$ when $n \geq 5$ and $e \geq 2$.

\end{itemize}

In order to compute $H_2(B^{(k)}(e,e,n),\mathbb{Z})$, we follow exactly the proof of \mbox{Theorem 6.4} in \cite{CalMarHomologyComputations} and use the same notations. We only point out the part of our proof that is different from \cite{CalMarHomologyComputations}. Define $v_i = [\tilde{t}_0,\tilde{t}_i] + [\tilde{s}_3,\tilde{t}_0] + [\tilde{s}_3,\tilde{t}_k] - [\tilde{s}_3,\tilde{t}_i] - [\tilde{s}_3, \tilde{t}_{i+k}]$ where $1 \leq i \leq e-1$. As in \cite{CalMarHomologyComputations}, we also have $H_2(B^{(k)}(e,e,n),\mathbb{Z}) = (K_1/d_3(C_1)) \oplus (K_2/d_3(C_2))$. We have $d_3[\tilde{s}_3,\tilde{t}_0, \tilde{t}_j] = v_j - v_{j+k} + v_k$ if $j \neq -k$ and $d_3[\tilde{s}_3,\tilde{t}_0, \tilde{t}_{-k}] =v_{-k} + v_k$. Denote $u_i = [\tilde{s}_3, \tilde{t}_0,\tilde{t}_i]$ for $1 \leq i \leq e-1$. We define a basis of $C_2$ as follows. For each coset of the subgroup of $\mathbb{Z}/e\mathbb{Z}$ generated by the class of $k$, say $\{\tilde{t}_x,\tilde{t}_{x+k}, \cdots, \tilde{t}_{x-k}\}$ such that $1 \leq x \leq e-1$, we define $w_{x+ik} = u_{x+ik} + u_{x+(i+1)k} + \cdots + u_{x-k}$ for $0 \leq i \leq e-1$, and when $x=0$, we define $w_{ik} = u_{ik} + u_{(i+1)k} + \cdots + u_{-k}$ for $1 \leq i \leq e-1$. Written on the $\mathbb{Z}$-basis $(w_{k}, w_{2k}, \cdots , w_{-k}, w_{1}, w_{1+k}, \cdots, w_{1-k}, \cdots, w_{x}, w_{x+k}, \cdots, w_{x-k}, \cdots)$ and $(v_k, v_{2k}, \cdots, v_{-k}, v_1, v_{1+k}, \cdots, v_{1-k}, \cdots, v_x, v_{x+k}, \cdots, v_{x-k}, \cdots )$, $d_3$ is in triangular form with $(e \wedge k) -1$ diagonal coefficients that are zero, all other diagonal coefficients are equal to $1$ except one of them that is equal to $e' = e / (e \wedge k)$. In this case, we have $H_2(B^{(k)}(e,e,3),\mathbb{Z}) = \mathbb{Z}^{(e \wedge k)-1} \times \mathbb{Z}/e'\mathbb{Z}$. The rest of the proof is essentially similar to the proof of Theorem 6.4 in  \cite{CalMarHomologyComputations}.

When $n=4$, we get $[\tilde{s}_4,\tilde{t}_i] + [\tilde{s}_4,\tilde{t}_{i+k}] \equiv [\tilde{s}_4,\tilde{t}_{k}] + [\tilde{s}_4,\tilde{t}_{0}]$ for every $i$ and since $2[\tilde{s}_4,\tilde{t}_i] \equiv 0$ for every $i$, we get $K_2/d_3(C_2) \simeq (\mathbb{Z}/2\mathbb{Z})^c$, where $c$ is the number of $[\tilde{s}_4,\tilde{t}_j]$ that appear in the solution of the congruence relations $[\tilde{s}_4,\tilde{t}_i] + [\tilde{s}_4,\tilde{t}_{i+k}] \equiv [\tilde{s}_4,\tilde{t}_{k}] + [\tilde{s}_4,\tilde{t}_{0}]$ and $2[\tilde{s}_4,\tilde{t}_i] \equiv 0$ in $\mathbb{Z}/e\mathbb{Z}$. We summarize our proof by providing the second integral homology group of $B^{(k)}(e,e,n)$ in the following proposition.

\begin{proposition}\label{PropH2ofBkeen}

Let $e \geq 2$, $n \geq 3$, $1 \leq k \leq e-1$, and $e' = e/ e \wedge k$.
\begin{itemize}
\item If $n=3$, we have $H_2(B^{(k)}(e,e,3),\mathbb{Z}) \simeq \mathbb{Z}^{(e \wedge k)-1} \times \mathbb{Z}/e'\mathbb{Z}$.
\item If $n=4$, we have $H_2(B^{(k)}(e,e,4),\mathbb{Z}) \simeq \mathbb{Z}^{(e \wedge k)-1} \times \mathbb{Z}/e'\mathbb{Z} \times (\mathbb{Z}/2\mathbb{Z})^{c}$, where $c$ is defined in the previous paragraph.
\item If $n \geq 5$, we have $H_2(B^{(k)}(e,e,n),\mathbb{Z}) \simeq \mathbb{Z}^{(e \wedge k)-1} \times \mathbb{Z}/e'\mathbb{Z} \times \mathbb{Z}/2\mathbb{Z}$.

\end{itemize}

\end{proposition}

Comparing this result with $H_2(B(e,e,n),\mathbb{Z})$, one can readily check that if $k \wedge e \neq 1$, $B^{(k)}(e,e,n)$ is not isomorphic to a complex braid group of type $B(e,e,n)$. Thus, we conclude this section by the following theorem.

\medskip

\begin{theorem}\label{TheoremBkIsomBiff}

$B^{(k)}(e,e,n)$ is isomorphic to $B(e,e,n)$ if and only if \mbox{$k\wedge e=1$.}

\end{theorem}

\medskip

In Appendix A of \cite{ThesisNeaime}, we provide a general code to compute the homology groups of order $r \geq 2$ of Garside structures by using the Dehornoy-Lafont complexes. In particular, we apply the code for the Garside monoids $B^{\oplus k}(e,e,n)$ that have been already implemented (see \cite{CHEVIEJMichel} and \cite{GAP3CPMichelNeaime}) and compute some cases of the integral homology groups of $B^{(k)}(e,e,n)$.

\section*{Acknowledgements}

This article is part of my doctoral thesis defended at Universit\'e de Caen Normandie. I thank my PhD supervisors Eddy Godelle and Ivan Marin for fruitful discussions.

\bibliographystyle{plain}
\bibliography{ArxivIntervalStructuresB(een)V2.bib}

\end{document}